\documentclass[reqno]{amsart}
\pdfoutput=1
\usepackage[x11names]{xcolor}
\usepackage{amsmath, amssymb}
\usepackage{amsthm}
\usepackage{mathtools}
\usepackage{hyperref}
\usepackage[shortlabels]{enumitem}
\usepackage[portrait, margin=1in]{geometry}

\hypersetup{
    colorlinks = true,
    linktoc = all,
    urlcolor = Blue3!50!Blue4,
    linkcolor = Blue3!50!Blue4,
    citecolor = Blue3!50!Blue4
}

\usepackage{tikz}
\usetikzlibrary{trees}
\usetikzlibrary{arrows.meta}
\usetikzlibrary{decorations.pathmorphing}
\tikzset{snaked/.style = {decorate, decoration=snake}}
\usetikzlibrary{positioning}
\usetikzlibrary{decorations.pathreplacing}
\tikzset{sdot/.style = {fill, circle, inner sep = 1.5pt}}
\usetikzlibrary{calc}

\newtheorem{theorem}{Theorem}[section]

\newtheorem{lemma}[theorem]{Lemma}
\newtheorem{prop}[theorem]{Proposition}

\newtheorem{claim}[theorem]{Claim}
\newtheorem*{theorem*}{Theorem}
\newtheorem*{corollary*}{Corollary}
\newtheorem*{lemma*}{Lemma}
\newtheorem*{prop*}{Proposition}

\newtheorem*{fact*}{Fact}
\newtheorem*{claim*}{Claim}

\theoremstyle{definition}
\newtheorem{example}[theorem]{Example}

\newtheorem*{example*}{Example}
\newtheorem*{defn*}{Definition}
\newtheorem*{remark*}{Remark}

\numberwithin{equation}{section}

\newcommand{\floor}[1]{\left\lfloor #1 \right\rfloor}
\newcommand{\ceil}[1]{\left\lceil #1 \right\rceil}
\newcommand{\sfloor}[1]{\lfloor #1 \rfloor}
\newcommand{\sceil}[1]{\lceil #1 \rceil}
\newcommand{\sdbold}[1]{\textbf{\textsf{#1}}}

\newcommand{\abs}[1]{\left\lvert #1 \right\rvert}
\newcommand{\sabs}[1]{\lvert #1 \rvert}

\newcommand{\EE}{\mathbb E}

\newcommand{\NN}{\mathbb N}

\newcommand{\QQ}{\mathbb Q}
\newcommand{\RR}{\mathbb R}

\newcommand{\ZZ}{\mathbb Z}

\newcommand{\cC}{\mathcal C}

\newcommand{\cP}{\mathcal P}

\newcommand{\cR}{\mathcal R}
\newcommand{\cS}{\mathcal S}
\newcommand{\cT}{\mathcal T}

\newcommand{\eps}{\varepsilon}

\newcommand{\case}[2]{\textcolor{MediumPurple3!80}{\sdbold{Case #1} (#2).}}
\newcommand{\casetext}[1]{\textcolor{MediumPurple3!80}{\sdbold{#1}}}

\title{Exponents in the local properties problem for difference sets have a gap at 2}
\author{Sanjana Das}
\date{January 19, 2025}
\address{Department of Mathematics, Massachusetts Institute of Technology, Cambridge, MA 02139}
\email{sanjanad@mit.edu}

\begin{document}

\begin{abstract}
    We study the local properties problem for difference sets: If we have a large set of real numbers and know that every small subset has many distinct differences, to what extent must the \emph{entire} set have many distinct differences? More precisely, we define $g(n, k, \ell)$ to be the minimum number of differences in an $n$-element set with the `local property' that every $k$-element subset has at least $\ell$ differences; we study the asymptotic behavior of $g(n, k, \ell)$ as $k$ and $\ell$ are fixed and $n \to \infty$. 

    The \emph{quadratic threshold} is the smallest $\ell$ (as a function of $k$) for which $g(n, k, \ell) = \Omega(n^2)$; its value is known when $k$ is even. In this paper, we show that for $k$ even, when $\ell$ is one below the quadratic threshold, we have $g(n, k, \ell) = O(n^c)$ for an absolute constant $c < 2$ --- i.e., at the quadratic threshold, the `exponent of $n$ in $g(n, k, \ell)$' jumps by a constant independent of $k$. 
\end{abstract}

\maketitle

\section{Introduction}

\subsection{Background on local properties problems}

There is a long history of studying problems with the following form: If we have a large object and we know that every small piece of it is `unstructured' in some sense, how unstructured must the \emph{entire} object be? In other words, to what extent can we go from a \emph{local} property about lack of structure to a \emph{global} one? 

Erd\H{o}s and Shelah \cite[Section V]{Erd75} initiated the study of a local properties problem for graphs. Here, the large object is an edge-coloring of a complete graph, and the small pieces are constant-sized induced subgraphs; we think of an edge-coloring as unstructured if it contains many distinct colors. To formalize this, we define $f(n, k, \ell)$ to be the minimum number of colors needed to edge-color $K_n$ such that every induced subgraph $K_k$ contains at least $\ell$ colors. This can be viewed as a generalization of Ramsey numbers, which correspond to the case $\ell = 2$ --- the $t$-color Ramsey number $r_t(k)$ is the smallest $n$ for which it is not possible to edge-color $K_n$ with $t$ colors such that every $K_k$ contains at least two colors, i.e., the smallest $n$ for which $f(n, k, 2) > t$. 

When studying $f(n, k, \ell)$, we typically think of $k$ and $\ell$ as constants, and study the asymptotic behavior of $f(n, k, \ell)$ as $n \to \infty$. One direction to approach this problem from is to search for thresholds: If we fix $k$ and increase $\ell$ from $1$ to $\binom{k}{2}$, then $f(n, k, \ell)$ will increase from $1$ to $\binom{n}{2}$, and we can ask for the value of $\ell$ at which $f(n, k, \ell)$ begins to exhibit a certain behavior, which we call the \emph{threshold} for that behavior. In \cite{EG97}, Erd\H{o}s and Gy\'arf\'as introduced the following thresholds.
\begin{itemize}
    \item The \emph{polynomial threshold} is the smallest $\ell$ (as a function of $k$) for which $f(n, k, \ell) = \Omega(n^\eps)$ for some $\eps > 0$ (possibly depending on $k$). 
    \item The \emph{linear threshold} is the smallest $\ell$ for which $f(n, k, \ell) = \Omega(n)$. 
    \item The \emph{superlinear threshold} is the smallest $\ell$ for which $f(n, k, \ell) = \omega(n)$. 
    \item The \emph{quadratic threshold} is the smallest $\ell$ for which $f(n, k, \ell) = \Omega(n^2)$. 
\end{itemize}

Erd\H{o}s and Gy\'arf\'as \cite{EG97} exactly determined the linear and quadratic thresholds --- they showed that the linear threshold is $\binom{k}{2} - k + 3$ and the quadratic threshold is $\binom{k}{2} - \sfloor{k/2} + 2$. They also showed that the polynomial threshold is at most $k$;  Conlon, Fox, Lee, and Sudakov \cite{CFLS15} later proved that this is tight. S\'ark\"ozy and Selkow \cite{SS01} showed that the superlinear threshold is at most $\binom{k}{2} - k + \sceil{\log k} + 3$.

Erd\H{o}s and Gy\'arf\'as \cite{EG97} also proved a general upper bound on $f(n, k, \ell)$ using a random construction; works including \cite{PS19, FPS20, BEHK23} have proved several families of lower bounds, which nearly match this upper bound in many cases. 

Erd\H{o}s \cite{Erd86} also posed a similar local properties problem for distinct distances. Here, the large object is a set of points in $\RR^2$, and we think of a set of points as unstructured if it spans many distinct distances. So we define $\phi(n, k, \ell)$ as the minimum number of distinct distances that $n$ points in $\RR^2$ can span, given that every $k$ points span at least $\ell$ distinct distances. 

We always have $\phi(n, k, \ell) \geq f(n, k, \ell)$ --- we can convert any configuration of points into a graph on those points where the color of an edge represents the distance between its two endpoints --- and in many regimes, this is the best known lower bound for $\phi(n, k, \ell)$. However, there are a few cases where we are able to make use of the geometric structure to prove better lower bounds: Fox, Pach, and Suk \cite{FPS18} proved that $\phi(n, k, \binom{k}{2} - k + 6) = \Omega(n^{8/7 - o(1)})$. (Meanwhile, we only know that $f(n, k, \binom{k}{2} - k + 6) = \Omega(n)$.) 

It is also possible to approach this problem in terms of thresholds (in particular, Fox, Pach, and Suk comment on the linear and quadratic thresholds in \cite[Section 1]{FPS18}), but such thresholds are much farther from being understood than those for $f(n, k, \ell)$, at least for general values of $k$. Erd\H{o}s also posed several problems about determining $\phi(n, k, \ell)$ for small values of $k$ and $\ell$; in particular, in \cite{Erd86} he asked whether $g(n, 4, 5) = \Omega(n^2)$. Recently Tao \cite{Tao24} proved that the answer is no, which means that when $k = 4$, the quadratic threshold is $6$. See \cite[Section 7]{She18} for more about this problem.

\subsection{A local properties problem for difference sets}

We study an arithmetic local properties problem, first described in \cite{PS19}. Here, our large object is a set of numbers, and we think of a set of numbers as unstructured if it contains many distinct differences. More formally, for a set $A \subseteq \RR$, we define the \emph{difference set} of $A$ to be \[A - A = \{\abs{a - b} \mid a, b \in A, \, a \neq b\}.\] We only include positive differences in $A - A$; this is nonstandard but more natural for this problem. 

We define $g(n, k, \ell)$ to be the minimum value of $\abs{A - A}$ over all $n$-element sets $A \subseteq \RR$ with the property that every $k$-element subset $A' \subseteq A$ satisfies $\sabs{A' - A'} \geq \ell$, which we refer to as the \emph{$(k, \ell)$-local property}. We think of $k$ and $\ell$ as constants, and study the asymptotic behavior of $g(n, k, \ell)$ as $n \to \infty$. (We use standard asymptotic notation throughout this paper; all asymptotic notation is as $n \to \infty$, and the implicit constants may depend on $k$, $\ell$, and any other relevant parameters.)

This problem can be viewed as a one-dimensional version of the local properties problem for distinct distances --- in particular, we have \begin{equation}
    g(n, k, \ell) \geq \phi(n, k, \ell) \geq f(n, k, \ell) \label{eqn:g-phi-f}
\end{equation}
(since any collection of real numbers $A \subseteq \RR$ can be viewed as a set of points in $\RR^2$). 

Note that any $m$-element set $A$ satisfies $m - 1 \leq \abs{A - A} \leq \binom{m}{2}$. This means that the problem is only interesting when $k - 1 \leq \ell \leq \binom{k}{2}$ (the $(k, \ell)$-local property is vacuous for smaller values of $\ell$, and impossible to satisfy for larger values of $\ell$). Also, if we increase $\ell$ from $k - 1$ to $\binom{k}{2}$, then $g(n, k, \ell)$ will increase from $n - 1$ to $\binom{n}{2}$ (as long as $k \geq 4$). So the most natural thresholds to consider are the superlinear and quadratic threshold, both introduced by Li \cite{Li22} (defined analogously to the corresponding thresholds for $f$). 
\begin{itemize}
    \item The \emph{superlinear threshold} is the smallest $\ell$ (as a function of $k$) for which $g(n, k, \ell) = \omega(n)$. 
    \item The \emph{quadratic threshold} is the smallest $\ell$ for which $g(n, k, \ell) = \Omega(n^2)$. 
\end{itemize}

\subsection{Previous work}

Because of \eqref{eqn:g-phi-f}, all lower bounds for $f$ or $\phi$ are also lower bounds for $g$. However, it turns out that we know significantly stronger lower bounds for $g$ than for $f$ or $\phi$. (In the opposite direction, apart from a few cases with $k \leq 4$, nearly all the nontrivial \emph{upper} bounds we know on $\phi$ come from upper bounds on $g$, while Erd\H{o}s and Gy\'arf\'as's random construction in \cite{EG97} gives much better upper bounds on $f$.)

The first lower bounds specific to $g$ were proven by Fish, Pohoata, and Sheffer \cite{FPS20}, who showed that for all $r \geq 2$ and $k$ divisible by $2r$, we have
\begin{equation}
    g\left(n, k, \binom{k}{2} - \frac{(r - 1)(r + 2)}{2}\binom{k/r}{2} + 1\right) = \Omega(n^{\frac{r}{r - 1} \cdot \frac{k - 2r}{k}}). \label{eqn:fps-lower}
\end{equation}
For example, when $r = 2$, this states that for $4 \mid k$ we have $g(n, k, k^2/4 + 1) = \Omega(n^{2 - 8/k})$. More generally, if we think of $r$ as fixed and $k$ as reasonably large, \eqref{eqn:fps-lower} gives a family of lower bounds on $g(n, k, \ell)$ for values of $\ell$ roughly between $7k^2/32$ and $k^2/4$. (For comparison, such values of $\ell$ are well below the linear threshold for $f$, which occurs at $\binom{k}{2} - k + 3$.)

Fish, Pohoata, and Sheffer also proved an upper bound for $g(n, k, \ell)$ for small values of $\ell$: They showed that for every $\eps > 0$, there exists $a > 0$ such that for all sufficiently large $k$, we have 
\begin{equation}
    g(n, k, ak(\log k)^{1/4 - \eps}) = n \cdot 2^{O(\sqrt{\log n})} = n^{1 + o(1)}. \label{eqn:fps-3ap-upper}
\end{equation} 
This comes from a result in additive combinatorics that for such values of $\ell$, any $k$-element set with fewer than $\ell$ differences must contain a $3$-AP (in this paper, we use `$k$-AP' as an abbreviation for `$k$-term arithmetic progression'). So any $3$-AP-free set satisfies the $(k, \ell)$-local property; the bound then follows by using the construction of $3$-AP-free sets due to Behrend \cite{Beh46}. 

Fish, Lund, and Sheffer \cite{FLS19} then proved the upper bound \[g\left(n, k, \frac{k^{\log_2 3} - 1}{2}\right) = O(n^{\log_2 3}).\] Their construction was roughly an affine $t$-cube (a set of the form $\{a + \eps_1d_1 + \cdots + \eps_td_t \mid \eps_i \in \{0, 1\}\}$ for fixed $a$, $d_1$, \ldots, $d_t$) with $t \approx \log_2 n$. 

Li \cite{Li22} then introduced and studied the superlinear and quadratic thresholds. She proved that the superlinear threshold is exactly $k$; the only $k$-element subsets that violate the $(k, k)$-local property are $k$-APs, so this result states that any set with only linearly many differences must contain a $k$-AP. 

Li also proved that if $k$ is even, then
\begin{equation}
    g\left(n, k, \frac{3k^2}{8} - \frac{3k}{4} + 2\right) = \Omega(n^2), \label{eqn:li-quadratic}
\end{equation}
which means the quadratic threshold is at most $3k^2/8 - 3k/4 + 2$. The idea of the proof is that a set with few distinct differences must contain a difference repeated many times; then if we take $k/2$ pairs of numbers with that difference, they form a $k$-element subset with at most $3k^2/8 - 3k/4 + 1$ differences, violating the $(k, \ell)$-local property for $\ell = 3k^2/8 - 3k/4 + 2$. 

Li also proved several `intermediate' lower bounds for $g(n, k, \ell)$ --- in particular, she showed that if $k$ is divisible by $8$, then 
\begin{equation}
    g\left(n, k, \frac{9k^2}{32} - \frac{9k}{16} + 5\right) = \Omega(n^{4/3}), \label{eqn:li-4-3-upper}
\end{equation}
and if $k \geq 8$ is a power of $2$, then
\begin{equation}
    g\left(n, k, \frac{k^{\log_2 3} + 1}{2}\right) = \Omega(n^{1 + \frac{2}{k - 2}}). \label{eqn:li-pow-of-2}
\end{equation}
The proofs of both bounds can be interpreted as showing that a set with few distinct differences must contain many congruent affine $t$-cubes, and a $k$-element set consisting of $k/2^t$ congruent affine $t$-cubes violates the $(k, \ell)$-local property for the chosen values of $\ell$. (For \eqref{eqn:li-4-3-upper} we take $t = 2$, and for \eqref{eqn:li-pow-of-2} we take $t = \log_2 k - 1$; we can view the proof of \eqref{eqn:li-quadratic} as the same argument with $t = 1$.) 

Finally, Li also proved an upper bound for small values of $\ell$, using a random construction: She showed that for every $a \geq 2$, for all sufficiently large $k$ we have 
\begin{equation}
    g(n, k, ak + 1) = O(n^{1 + \frac{a^2 + 1}{k}}). \label{eqn:li-upper-small-l}
\end{equation} 

In \cite{Das23}, the author further studied the quadratic threshold. For $k$ even, they determined it exactly --- they showed that we have
\begin{align}
    g\left(n, k, \frac{k^2}{4} + 1\right) &= \Omega(n^2), \label{eqn:das-quadratic-lower} \\
    g\left(n, k, \frac{k^2}{4}\right) &= o(n^2), \label{eqn:das-quadratic-upper}
\end{align}
which means the quadratic threshold is exactly $k^2/4$. For $k$ odd, they determined the quadratic threshold up to a constant-length window --- they showed that
\begin{align}
    g\left(n, k, \frac{(k + 1)^2}{4}\right) &= \Omega(n^2), \label{eqn:das-odd-quadratic-lower} \\
    g\left(n, k, \frac{(k + 1)^2}{4} - 4\right) &= o(n^2). \label{eqn:das-odd-quadratic-upper}
\end{align}
The lower bounds \eqref{eqn:das-quadratic-lower} and \eqref{eqn:das-odd-quadratic-lower} come from showing that a set with few distinct differences must contain a \emph{sum} repeated many times, and a $k$-element set formed by taking many pairs with the same sum violates the $(k, \ell)$-local property. The upper bounds \eqref{eqn:das-quadratic-upper} and \eqref{eqn:das-odd-quadratic-upper} come from a random construction. 

In \cite{Das23}, the author also proved a family of `intermediate' lower and upper bounds --- they showed that 
\begin{equation}
    g\left(n, k, \frac{3^{t - 1}}{4^t}k^2 + \frac{3^{t - 1} + 1}{2}\right) = \Omega(n^{1 + \frac{1}{2^t - 1}}) \label{eqn:das-interm-lower}
\end{equation} 
for all integers $t \geq 1$, and
\begin{equation}
    g\left(n, k, \ceil{\frac{(c - 1)(k - 1)}{c}}^2\right) = o(n^c) \label{eqn:das-interm-upper}
\end{equation}
for all $1 < c \leq 2$. The proof of the lower bound \eqref{eqn:das-interm-lower} combines ideas from Li's proof of \eqref{eqn:li-4-3-upper} and \eqref{eqn:li-pow-of-2} with ones from the proof of \eqref{eqn:das-quadratic-lower} --- a set with few distinct differences must contain $k/2^{t - 1}$ disjoint affine $(t - 1)$-cubes whose centers form pairs with equal sums. The proof of the upper bound \eqref{eqn:das-interm-upper} again comes from a random construction. 

If we think of $t$ and $c$ as fixed and $k$ as reasonably large, the values of $\ell$ in both \eqref{eqn:das-interm-lower} and \eqref{eqn:das-interm-upper} are quadratic in $k$, and the exponents of $n$ are constants greater than $1$. So one interpretation of these bounds is that for all $1 < c \leq 2$, the \emph{threshold for $\Omega(n^c)$} --- i.e., the smallest $\ell$ for which $g(n, k, \ell) = \Omega(n^c)$ --- is `quadratic in $k$.' (More precisely, for each $1 < c \leq 2$, there exist $a_1, a_2 > 0$ such that for all large $k$, the threshold for $\Omega(n^c)$ is between $a_1k^2$ and $a_2k^2$. Furthermore, these constants can be chosen such that as $c \to 1$, we have $a_1, a_2 \to 0$.)

\begin{figure}[ht]
    \centering
    \begin{tikzpicture}[scale = 1.1]
        \draw [Latex-Latex, gray!70] (-0.5, 0) -- (12.5, 0);
        \draw [Latex-Latex, gray!70] (0, -0.5) -- (0, 4.25);
        
        \foreach \i\j in {1/2, 1/4, 3/16, 9/64, 27/256, 2/9} {
            \draw [gray!70] (24*\i/\j, 0.15) -- (24*\i/\j, -0.15) node [anchor = north] {$\frac{\i}{\j}$};
        }
        \foreach \i in {2, 1} {
            \draw [gray!70] (0.15, {3.5*(\i - 1) + 0.25}) -- (-0.15, {3.5*(\i - 1) + 0.25}) node [anchor = east] {$\i$};
        }
        \foreach \i\j in {3/2, 4/3, 8/7} {
            \draw [gray!70] (0.15, {3.5*(\i/\j - 1) + 0.25}) -- (-0.15, {3.5*(\i/\j - 1) + 0.25}) node [anchor = east] {$\frac{\i}{\j}$};
        }
        \draw [gray!50, dashed] (0, 0.25) -- (12, 0.25);
        \draw [gray!50, dashed] (0, 3.75) -- (12, 3.75);
        \draw [gray!50, dashed] (12, 0) -- (12, 3.75);
        \begin{scope}[yshift = 0.25cm, yscale = 3.5, yshift = -1cm, xscale = 24]
            \coordinate (a0) at (1/4, 1.987);
            \coordinate (b0) at (1/4, 3/2);
            \coordinate (a1) at (2/9, 3/2);
            \coordinate (b1) at (2/9, 4/3);
            \foreach \i in {2, 3, ..., 10} {
                \coordinate (a\i) at ({1/3*(3/4)^\i}, {1 + 1/(2^\i - 1)});
                \coordinate (b\i) at ({1/3*(3/4)^\i}, {1 + 1/(2*2^\i - 1)});
            }
            \coordinate (a11) at (0, 1);
            \fill [DarkSlateGray3!50!MediumPurple3!10] (a0) -- (b0) -- (a1) -- (b1) -- (a2) -- (b2) -- (a3) -- (b3) -- (a4) -- (b4) -- (a5) -- (b5) -- (a6) -- (b6) -- (a7) -- (b7) -- (a8) -- (b8) -- (a9) -- (b9) -- (a10) -- (b10) -- (a11) [domain = 1:2, samples = 20] plot({(\x - 1)^2/(\x^2)}, \x);
            \draw [MediumPurple3!80, very thick] (1/2, 1.987) -- (a0);
            \draw [MediumPurple4, very thick] (a0) -- (b0) -- (a1);
            \draw [MediumPurple3!80, very thick] (a1) -- (b1) -- (a2) -- (b2) -- (a3) -- (b3) -- (a4) -- (b4) -- (a5) -- (b5) -- (a6) -- (b6) -- (a7) -- (b7) -- (a8) -- (b8) -- (a9) -- (b9) -- (a10) -- (b10) -- (a11);
            \draw [DarkSlateGray3, very thick] (1/2, 2) -- (1/4, 2) [domain = 2:1, samples = 20] plot({(\x - 1)^2/(\x^2)}, \x);
        \end{scope}
    \end{tikzpicture}
    \caption{A plot of the bounds we know on $g(n, k, \ell)$ in the regime where $\ell$ is quadratic in $k$, where the $x$-axis depicts the coefficient of $k^2$ in $\ell$ and the $y$-axis depicts the exponent of $n$ in the bound --- a point $(a, c)$ means that for $\ell \approx ak^2$ we have a bound of roughly $n^c$. The purple line represents lower bounds --- for $2/9 < a < 1/4$ the best lower bound comes from \eqref{eqn:fps-lower} with $r = 3$, and for other values of $a$ the best lower bound comes from \eqref{eqn:das-interm-lower}. The blue line represents upper bounds, which come from \eqref{eqn:das-interm-upper}.} \label{fig:previous-bounds}
\end{figure}
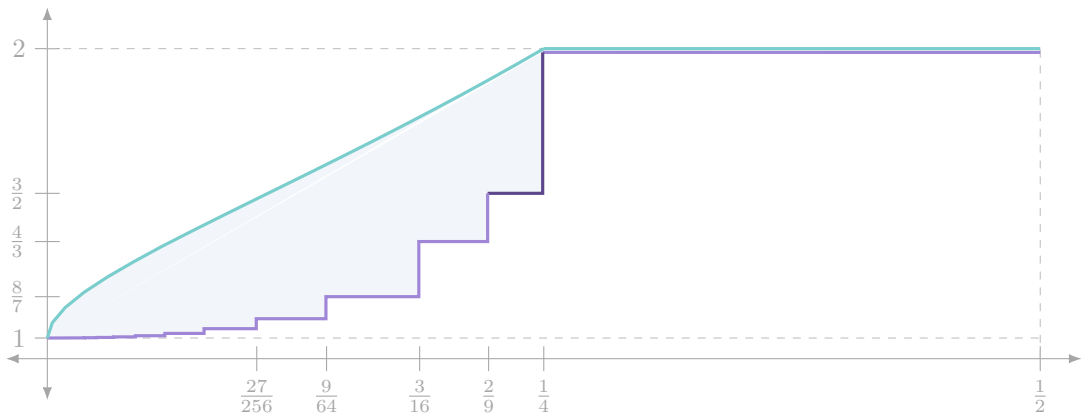

\subsection{Our question and result}

Once we understand a threshold for $g$, it is natural to consider what the transition in the behavior of $g$ `looks like' at that threshold. For example, Li \cite{Li22} showed that the superlinear threshold is $k$, so it is natural to ask how \emph{much} faster than linear $g(n, k, \ell)$ grows once $\ell$ crosses this threshold. Behrend's construction of $3$-AP-free (and therefore $k$-AP-free) sets with $n^{1 + o(1)}$ differences shows that for any $k \geq 3$, we have $g(n, k, k) = n^{1 + o(1)}$. So when we cross the superlinear threshold, $g(n, k, \ell)$ changes from a function which is linear in $n$ to one which is not linear, but still only $n^{1 + o(1)}$. Furthermore, \eqref{eqn:fps-3ap-upper} means that when $k$ is large, $g(n, k, \ell)$ remains $n^{1 + o(1)}$ for a reasonably large range of $\ell$. (This question was also the motivation behind Li's upper bound \eqref{eqn:li-upper-small-l}.)

In this paper, we study this question for the quadratic threshold. When $k$ is even, the quadratic threshold is $k^2/4 + 1$, so it is natural to ask what the behavior of $g(n, k, k^2/4)$ looks like --- we know that it is subquadratic, but \emph{how} far from quadratic is it? So far, the best lower bound we know on $g(n, k, k^2/4)$ (for reasonably large $k$) comes from the bound \eqref{eqn:fps-lower} of Fish, Pohoata, and Sheffer \cite{FPS20} (with $r = 3$), which gives 
\begin{equation}
    g\left(n, k, \frac{k^2}{4}\right) = \Omega(n^{3/2 - 9/k}). \label{eqn:lower-below-qt}
\end{equation}  
Meanwhile, a more careful analysis of the proof of \eqref{eqn:das-quadratic-upper} from \cite{Das23} would give the quantitative bound
\begin{equation}
    g\left(n, k, \frac{k^2}{4}\right) = O(n^{2 - 2/k + o(1)}). \label{eqn:upper-below-qt}
\end{equation}
In fact, this bound can be deduced directly from \eqref{eqn:das-interm-upper} --- if we plug any $c > 2 - 2/k$ into \eqref{eqn:das-interm-upper}, the value of $\ell$ it gives is still $k^2/4$.

For any \emph{fixed} $k$, the exponent of $n$ in \eqref{eqn:upper-below-qt} is bounded away from $2$; so it is not the case that $g(n, k, k^2/4) = n^{2 - o(1)}$. However, the behaviors of the lower and upper bounds \eqref{eqn:lower-below-qt} and \eqref{eqn:upper-below-qt} are qualitatively very different as $k$ becomes large. In \eqref{eqn:lower-below-qt} the exponent of $n$ is at most $3/2$, which is bounded away from $2$; in \eqref{eqn:upper-below-qt} it is at least $2 - 2/k$, which becomes arbitrarily close to $2$ as $k$ grows. So it is natural to ask which behavior is correct. 

In this paper, we show that the qualitative behavior of the lower bound \eqref{eqn:lower-below-qt} is correct --- the exponent of $n$ in $g(n, k, k^2/4)$ is bounded away from $2$ by a constant independent of $k$. (One can formalize the notion of the `exponent of $n$ in $g(n, k, \ell)$' by considering $\liminf_{n \to \infty} \log_n g(n, k, \ell)$ or $\limsup_{n \to \infty} \log_n g(n, k, \ell)$.)

\begin{theorem} \label{thm:main}
    There is an absolute constant $c < 2$ such that for all even $k$, we have \[g\left(n, k, \frac{k^2}{4}\right) = O(n^c).\]   
\end{theorem}

Our proof obtains the constant $c = 2 - 2^{-29}$. (We did not optimize this.)

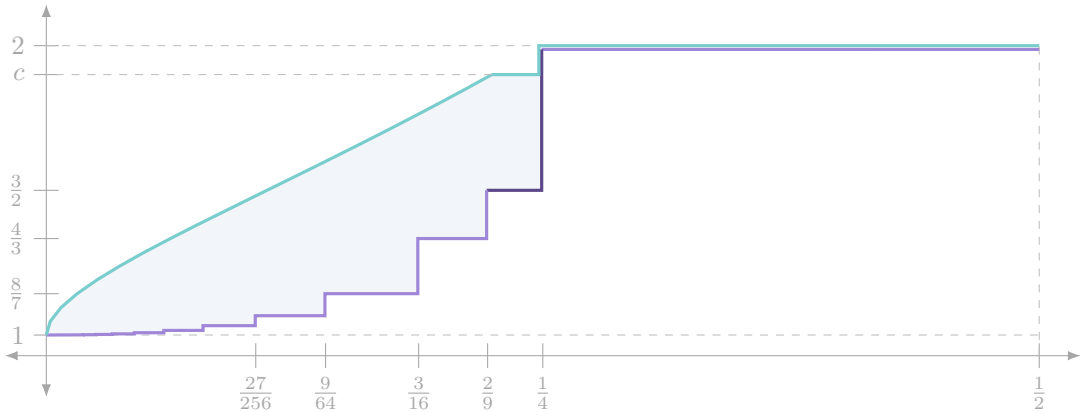
\begin{figure}[ht]
    \centering
    \begin{tikzpicture}[scale = 1.1]
        \draw [Latex-Latex, gray!70] (-0.5, 0) -- (12.5, 0);
        \draw [Latex-Latex, gray!70] (0, -0.5) -- (0, 4.25);
        
        \foreach \i\j in {1/2, 1/4, 3/16, 9/64, 27/256, 2/9} {
            \draw [gray!70] (24*\i/\j, 0.15) -- (24*\i/\j, -0.15) node [anchor = north] {$\frac{\i}{\j}$};
        }
        \foreach \i in {2, 1} {
            \draw [gray!70] (0.15, {3.5*(\i - 1) + 0.25}) -- (-0.15, {3.5*(\i - 1) + 0.25}) node [anchor = east] {$\i$};
        }
        \foreach \i\j in {3/2, 4/3, 8/7} {
            \draw [gray!70] (0.15, {3.5*(\i/\j - 1) + 0.25}) -- (-0.15, {3.5*(\i/\j - 1) + 0.25}) node [anchor = east] {$\frac{\i}{\j}$};
        }
        \draw [gray!70] (0.15, {3.5 * 0.9 + 0.25}) -- (-0.15, {3.5 * 0.9 + 0.25}) node [anchor = east] {$c$};
        \draw [gray!50, dashed] (0, 0.25) -- (12, 0.25);
        \draw [gray!50, dashed] (0, 3.75) -- (12, 3.75);
        \draw [gray!50, dashed] (12, 0) -- (12, 3.75);
        \begin{scope}[yshift = 0.25cm, yscale = 3.5, yshift = -1cm, xscale = 24]
            \coordinate (a0) at (1/4, 1.987);
            \coordinate (b0) at (1/4, 3/2);
            \coordinate (a1) at (2/9, 3/2);
            \coordinate (b1) at (2/9, 4/3);
            \foreach \i in {2, 3, ..., 10} {
                \coordinate (a\i) at ({1/3*(3/4)^\i}, {1 + 1/(2^\i - 1)});
                \coordinate (b\i) at ({1/3*(3/4)^\i}, {1 + 1/(2*2^\i - 1)});
            }
            \coordinate (a11) at (0, 1);
            \fill [DarkSlateGray3!50!MediumPurple3!10] (1/4, 1.9) -- (b0) -- (a1) -- (b1) -- (a2) -- (b2) -- (a3) -- (b3) -- (a4) -- (b4) -- (a5) -- (b5) -- (a6) -- (b6) -- (a7) -- (b7) -- (a8) -- (b8) -- (a9) -- (b9) -- (a10) -- (b10) -- (a11) [domain = 1:1.9, samples = 20] plot({(\x - 1)^2/(\x^2)}, \x) -- (1/4, 1.9);
            \draw [MediumPurple3!80, very thick] (1/2, 1.987) -- (a0);
            \draw [MediumPurple4, very thick] (a0) -- (b0) -- (a1);
            \draw [MediumPurple3!80, very thick] (a1) -- (b1) -- (a2) -- (b2) -- (a3) -- (b3) -- (a4) -- (b4) -- (a5) -- (b5) -- (a6) -- (b6) -- (a7) -- (b7) -- (a8) -- (b8) -- (a9) -- (b9) -- (a10) -- (b10) -- (a11);
            \draw [DarkSlateGray3, very thick] (1/2, 2) -- (0.248, 2) -- (0.248, 1.9) -- (0.81/3.61, 1.9) [domain = 1.9:1, samples = 20] plot({(\x - 1)^2/(\x^2)}, \x);
            \draw [gray!50, dashed] (0.81/3.61, 1.9) -- (0, 1.9);
        \end{scope}
    \end{tikzpicture}
    \caption{A version of Figure \ref{fig:previous-bounds} incorporating Theorem \ref{thm:main} (the value of $c$ is not to scale). There is a `gap' on the $y$-axis between $c$ and $2$ --- the exponent of $n$ in $g(n, k, \ell)$ can never lie in this range when $k$ is even.}
\end{figure}

Interestingly, the question of what happens immediately below the quadratic threshold has also been studied for the \emph{graph} local properties problem, and there the answer is the opposite. The quadratic threshold is $\binom{k}{2} - \sfloor{k/2} + 2$. Immediately below this threshold, Erd\H{o}s and Gy\'arf\'as \cite{EG97} proved an upper bound of \[f\left(n, k, \binom{n}{2} - \floor{\frac{k}{2}} + 1\right) = O(n^{2 - 4/k}),\] while Fish, Pohoata, and Sheffer \cite{FPS20} proved a lower bound of \[f\left(n, k, \binom{k}{2} - \frac{k}{2} + 1\right) = \Omega(n^{2 - 8/k})\] when $8 \mid k$. In both bounds, the exponent of $n$ is less than $2$ for any fixed $k$, but grows arbitrarily close to $2$ as $k$ grows. So the transitions that $g$ and $f$ display at their quadratic thresholds are quite different --- for $g$ there is a constant-sized gap between the exponents of $n$ in $g(n, k, \ell)$ at and below the quadratic threshold, while for $f$ the gap becomes arbitrarily small as $k$ grows (at least when $8 \mid k$). 

For odd $k$, we obtain a similar improvement to the bound \eqref{eqn:das-odd-quadratic-upper} on $g(n, k, (k + 1)^2/4 - 4)$. 

\begin{prop} \label{prop:odd-k}
    There is an absolute constant $c < 2$ such that for all odd $k$, we have \[g\left(n, k, \frac{(k + 1)^2}{4} - 4\right) = O(n^c).\]   
\end{prop}

Unlike for even $k$, we do not know the exact quadratic threshold for odd $k$ --- we know that $g(n, k, (k + 1)^2/4) = \Omega(n^2)$, but we do not understand the behavior of $g(n, k, \ell)$ when $\ell$ is between these two values. However, Proposition \ref{prop:odd-k} does guarantee that the exponent of $n$ in $g(n, k, \ell)$ jumps by a constant independent of $k$ when we decrease $\ell$ from the quadratic threshold to \emph{four} below it (for even $k$, we could make the same statement with one in place of four). 

\subsection{Overview} \label{subsec:overview}

We now give an overview of the ideas of the proof of Theorem \ref{thm:main}. (The proof of Proposition \ref{prop:odd-k} uses the same ideas.)

First, it is often useful to think about whether a set satisfies the $(k, \ell)$-local property in terms of the `configurations of equal differences' formed by its $k$-element subsets. More precisely, given $k$ numbers, we can write down a (minimal) system of equations that keeps track of which differences among them are equal --- for example, $\{1, 2, 5, 6, 9\}$ corresponds to the system $\{x_1 - x_2 = x_3 - x_4, \, x_1 - x_3 = x_3 - x_5\}$. We refer to such a system of equations as a \emph{$k$-configuration}. We can figure out the number of distinct differences among $k$ numbers just by looking at the $k$-configuration they form; so a set satisfies the $(k, \ell)$-local property if and only if it avoids all $k$-configurations which have fewer than $\ell$ distinct differences. 

We will first discuss the ideas behind the proof of \eqref{eqn:das-quadratic-upper} (the weaker bound $g(n, k, k^2/4) = o(n^2)$) in \cite{Das23}, which the proof of Theorem \ref{thm:main} builds on. Our goal is to obtain a set $A$ with $\abs{A} = n$ and $\abs{A - A} = o(n^2)$ which satisfies the $(k, k^2/4)$-local property. We do so via a random construction. Roughly, we start by taking a random subset of $\{1, 2, \ldots, o(n^2)\}$ of size a bit bigger than $n$. Some $k$-configurations are expected to appear very few times in this random subset, so we can eliminate all their appearances using the alteration method. We refer to $k$-configurations that we can avoid in this way as \emph{$2$-bad}, and the remaining ones as \emph{$2$-good}. (We are suppressing a few technical details here --- the actual construction and the description of $2$-good and $2$-bad $k$-configurations are slightly more complicated --- but this is the main idea.)

It then suffices to show that every $2$-good $k$-configuration has at least $k^2/4$ distinct differences. The idea is that the number of times a $k$-configuration is expected to appear in our random subset is controlled by the number of times it appears in the ground set $\{1, 2, \ldots, o(n^2)\}$, which is controlled by the number of linearly independent equations it has. So a $2$-good $k$-configuration cannot have too many independent equations, and we can use this to show that it cannot have too many `repeated' differences (so it must have many distinct differences). The value $k^2/4$ is tight --- the $k$-configuration \[\{x_1 + x_2 = x_3 + x_4 = \cdots = x_{k - 1} + x_k\},\] which we call a \emph{star of size $k$}, is $2$-good and has exactly $k^2/4$ distinct differences. 

In order to prove Theorem \ref{thm:main}, we wish to improve the bound of $o(n^2)$ to $O(n^c)$; so we now take a random subset of $\{1, 2, \ldots, O(n^c)\}$. We refer to the $k$-configurations that we can avoid using the alteration method as \emph{$c$-bad}, and the remaining ones as \emph{$c$-good}. Being $c$-good is a somewhat weaker condition than being $2$-good. However, it turns out that we can still show that every $c$-good $k$-configuration has at least $k^2/4$ distinct differences --- i.e., the star of size $k$ is still the one with the fewest differences. Proving this statement is the bulk of the argument. For this, we use a stability-type argument, which consists of three steps.
\begin{enumerate}[(1)]
    \item \label{step:baseline} First, by directly reusing the ideas from the argument that every $2$-good $k$-configuration has at least $k^2/4$ distinct differences, we obtain a slightly weaker bound for $c$-good $k$-configurations (one that differs from $k^2/4$ by a constant factor depending on $2 - c$).
    \item \label{step:stability} Next, we show that any $k$-configuration whose number of distinct differences is close to this weaker bound must be close to a star of size $k$ --- more precisely, any $k$-configuration within a $(1 - \eps)$-factor of the weaker bound must contain a star of size $(1 - \eps')k$ (where $\eps' \to 0$ as $\eps \to 0$). 
    
    If $2 - c$ is sufficiently small with respect to $\eps$, then any $k$-configuration which is not within a $(1 - \eps)$-factor of the weaker bound will have fewer than $k^2/4$ distinct differences. So it now suffices to consider $k$-configurations which contain a huge star. 
    \item \label{step:huge-star} Finally, we show that among $c$-good $k$-configurations which contain a huge star, the one with the fewest differences is the star of size $k$. Intuitively, we can imagine constructing such a $k$-configuration by starting with a huge star and adding a few extra equations. Requiring the $k$-configuration to be $c$-good \emph{a priori} allows us to add a few more equations than we could if we required it to be $2$-good. However, we show roughly that even with these extra equations, there is nothing better to do with the remaining variables (i.e., the ones not in the huge star we started with) than adding them to the star; and once our star has all $k$ variables, we show that we cannot add any more equations on top of it without violating $c$-goodness. 
\end{enumerate}

The structure of the paper is as follows. In Section \ref{sec:setup}, we flesh out the concept of $k$-configurations and define which ones are $c$-good and $c$-bad. In Section \ref{sec:random-constr}, we describe the random construction and show that it avoids all $c$-bad $k$-configurations. (This random construction is almost the same as the one in \cite{Das23}, but we use a slightly stronger definition of $c$-goodness here for technical reasons, so we need to tweak the construction a bit.) In Section \ref{sec:backbone}, we present the `backbone' of the argument that every $c$-good $k$-configuration has at least $k^2/4$ distinct differences; this argument involves a handful of lemmas whose proofs are fairly technical, and we present these proofs in Sections \ref{sec:observations}--\ref{sec:huge-star}. This completes the proof of Theorem \ref{thm:main}; in Section \ref{sec:odd-k}, we explain how to adapt the argument to prove Proposition \ref{prop:odd-k}.  

\section{Setup} \label{sec:setup}

In this section, we formalize the notion of $k$-configurations and define which $k$-configurations we consider $c$-good and $c$-bad.

\subsection{Conventions}

Throughout this paper, we will work with systems of linear equations in $k$ variables $x_1$, \ldots, $x_k$; we first fix a few conventions regarding such systems. 

We always work over the field $\QQ$. We consider linear equations to be defined up to rearrangement, but not scaling. For example, we consider $x_1 - x_2 = x_3 - x_4$ and $x_1 + x_4 = x_2 + x_3$ to be the same equation, but $x_1 - x_2 = 0$ and $2x_1 - 2x_2 = 0$ to be different equations.

Given a linear equation $(*)$, we define its \emph{content} to be an expression $*$ such that $(*)$ is the equation ${*} = 0$. (For any nontrivial equation $(*)$ --- i.e., one that is not the equation $0 = 0$ --- there are two ways to define its content; we choose one arbitrarily.) For example, the content of the equation $x_1 - x_2 = x_3 - x_4$ is either $x_1 - x_2 - x_3 + x_4$ or $-x_1 + x_2 + x_3 - x_4$.

We say linear equations $(*_1)$, \ldots, $(*_t)$ are \emph{independent} if their contents $*_1$, \ldots, $*_t$ are linearly independent. We say a collection of linear equations $\cT = \{(*_1), \ldots, (*_t)\}$ \emph{implies} an equation $(*)$ if $(*)$ can be written as a linear combination of $(*_1)$, \ldots, $(*_t)$, i.e., 
\begin{equation}
    {*} = c_1{*_1} + \cdots + c_t{*_t} \label{eqn:imply-coeffs}
\end{equation} 
for some $c_1, \ldots, c_t \in \QQ$. We say $\cT$ \emph{minimally implies} $(*)$ if $\cT$ implies $(*)$ but no proper subset of $\cT$ does; equivalently, $\cT$ minimally implies $(*)$ if $\cT$ is independent and the coefficients $c_1$, \ldots, $c_t$ in \eqref{eqn:imply-coeffs} are nonzero. 

\subsection{Definition of \texorpdfstring{$k$}{k}-configurations}

We now define $k$-configurations and formalize some of the notions informally referred to in the proof overview in Subsection \ref{subsec:overview} (e.g., what it means for a $k$-configuration to `have' a certain number of distinct differences). 

We define a \emph{difference equality} to be a nontrivial equation of the form \[x_{i_1} - x_{i_2} = x_{i_3} - x_{i_4}\] for some (not necessarily distinct) indices $i_1, \ldots, i_4 \in [k]$. We define a \emph{$k$-configuration} to be a system of linear equations in $x_1$, \ldots, $x_k$ where each equation is a difference equality. For convenience, we consider two $k$-configurations to be the same if they produce equivalent systems of equations --- for example, we consider \[\{x_1 - x_2 = x_3 - x_4, \, x_1 - x_2 = x_5 - x_6\} \quad \text{and} \quad \{x_1 - x_2 = x_3 - x_4, \, x_3 - x_4 = x_5 - x_6\}\] to be the same, and we will typically write this $k$-configuration as $\{x_1 - x_2 = x_3 - x_4 = x_5 - x_6\}$. 

We define the $k$-configuration \emph{formed} by $k$ numbers $a_1$, \ldots, $a_k$ to be the $k$-configuration consisting of all the difference equalities that $(a_1, \ldots, a_k)$ satisfies. For example, the $k$-configuration formed by $1$, $2$, $4$, $5$, $9$, $10$ (given in that order) is $\{x_1 - x_2 = x_3 - x_4 = x_5 - x_6, \, x_1 - 2x_4 + x_5 = 0\}$.

We will only be interested in $k$-configurations formed by $k$ distinct numbers. So we say a $k$-configuration $\cC$ is \emph{invalid} if it implies $x_i = x_j$ for some $i \neq j$, and \emph{valid} otherwise. If $a_1$, \ldots, $a_k$ are distinct, then the $k$-configuration they form must be valid. 

For a pair $(i, j) \in [k]^2$ with $i > j$, we say a $k$-configuration $\cC$ \emph{certifies} $(i, j)$ if $\cC$ implies \[x_i - x_j = x_{i'} - x_{j'}\] for some $(i', j') \in [k]^2$ where either $i', j' < i$, or $j' = i$ and $i' < j$. 

The reason for this definition is as follows: Suppose that we want to understand the number of distinct differences among $k$ distinct numbers $a_1$, \ldots, $a_k$. Then we can imagine going through all possible differences $\abs{a_2 - a_1}$, $\abs{a_3 - a_1}$, $\abs{a_3 - a_2}$, $\abs{a_4 - a_1}$, $\abs{a_4 - a_2}$, $\abs{a_4 - a_3}$, \ldots, in that order, and recording whether each difference is repeated or new (i.e., whether we have already seen it or not). If we let $\cC$ be the $k$-configuration formed by $a_1$, \ldots, $a_k$, then a difference $\abs{a_i - a_j}$ is repeated if and only if $\cC$ certifies $(i, j)$. So the number of distinct differences among $a_1$, \ldots, $a_k$ is the number of pairs $(i, j)$ that $\cC$ does \emph{not} certify, or $\binom{k}{2}$ minus the number of pairs that it \emph{does} certify. 

In particular, this means that if $\cC$ is a valid $k$-configuration, then permuting the indices of its variables does not affect the number of pairs it certifies (we can always find $k$ distinct numbers $a_1$, \ldots, $a_k$ that form $\cC$; permuting indices in $\cC$ corresponds to reordering these numbers, which of course does not affect the number of distinct differences among them). 

\begin{example} \label{ex:star}
    For $2p \leq k$, we say a \emph{star of size $2p$} is a $k$-configuration \[\{x_{i_1} + x_{i_2} = x_{i_3} + x_{i_4} = \cdots = x_{i_{2p - 1}} + x_{i_{2p}}\}\] for distinct $i_1, \ldots, i_{2p} \in [k]$. To compute the number of pairs certified by a star of size $2p$, we can rename its variables so that it becomes $\{x_1 + x_2 = x_3 + x_4 = \cdots = x_{2p - 1} + x_{2p}\}$. This $k$-configuration certifies $(2i, 2j - 1)$ and $(2i, 2j)$ for all $j < i \leq p$, since it implies \[x_{2i} - x_{2j - 1} = x_{2j} - x_{2i - 1} \quad \text{and} \quad x_{2i} - x_{2j} = x_{2j - 1} - x_{2i - 1}.\] So in total, it certifies $2 + 4 + 6 + \cdots + 2(p - 1) = p^2 - p$ pairs. 

    \begin{figure}[ht]
        \begin{tikzpicture}[scale = 0.7]
            \foreach \i\j\k\l\m\n in {1/2/-2.5/0/left/right, 3/4/-2/1.2/above left/below right, 5/6/-0.2/-2/below/above, 7/8/1.5/1.9/above right/below left} {
                \node [sdot, label = \m: {$x_{\i}$}] (\i) at (\k, \l) {};
                \node [sdot, label = \n: {$x_{\j}$}] (\j) at (-\k, -\l) {};
                \draw [gray!30, very thick] (\i) -- (\j);
            }
            \foreach \i\j\k\l\c\p in {3/7/4/8/MediumPurple3/0.1} {
                \filldraw [\c, draw opacity = 0.25, fill opacity = \p] (\i.center) -- (\j.center) -- (\k.center) -- (\l.center) -- cycle;
            }
        \end{tikzpicture}
        \caption{A star of size $8$, where we depict a difference equality $x_{i_1} - x_{i_2} = x_{i_3} - x_{i_4}$ by placing $x_{i_1}$, $x_{i_2}$, $x_{i_4}$, $x_{i_3}$ in a parallelogram. A star of size $2p$ then consists of $p$ pairs with the same midpoint, which is the reason for the name. The shaded parallelogram shows that this star certifies $(8, 3)$ and $(8, 4)$.}
    \end{figure}
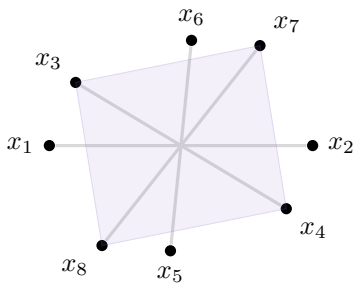
\end{example}

\subsection{Good and bad \texorpdfstring{$k$}{k}-configurations}

We will now define which $k$-configurations are $c$-good and $c$-bad. 

First, we say a $k$-configuration $\cC$ is \emph{collinearity-inducing} if it implies an equation containing exactly three variables, and \emph{collinearity-free} otherwise. (The reason for this name is that such an equation must be of the form $\alpha x_{i_1} + \beta x_{i_2} + \gamma x_{i_3} = 0$ where $\alpha + \beta + \gamma = 0$ --- the coefficients of a difference equality sum to $0$, so the same is true of any equation implied by a $k$-configuration --- and if we think of $x_1$, \ldots, $x_k$ as points, then such an equation would be saying that three of them are collinear.)

Given $1 < c \leq 2$, we say $\cC$ is \emph{$c$-heavy} if it implies a collection of $t$ independent equations which in total contain fewer than $ct + 1$ variables (for any $t \geq 1$); otherwise we say $\cC$ is \emph{$c$-light}.

Finally, we say $\cC$ is \emph{$c$-good} if it is valid, collinearity-free, and $c$-light; otherwise we say $\cC$ is \emph{$c$-bad}. Note that if $\cC$ is $c$-good, then every difference equality implied by $\cC$ contains four distinct variables (a difference equality containing two variables would be invalid, and one containing three variables would be collinearity-inducing). 

\begin{example} \label{ex:c-good}
    To illustrate these definitions, we consider a few examples with $c = 2$. (The value of $k$ in these examples is not important, as long as it is not smaller than any of the indices used.)
    \begin{enumerate}[(a)]
        \item \label{ex-item:2-eqns-4-vars} The $k$-configuration $\{x_1 - x_2 = x_3 - x_4, \, x_1 + x_2 = x_3 + x_4\}$ is invalid, as it implies $x_1 = x_3$. 
        \item \label{ex-item:collinearity} The $k$-configuration $\{x_1 - x_2 = x_3 - x_4, \, x_1 + x_2 = x_3 + x_5\}$ is valid and $2$-light. However, it is collinearity-inducing, as it implies $2x_2 - x_4 - x_5 = 0$. 
        
        \begin{figure}[ht]
            \begin{tikzpicture}
                \node [sdot, label = above: {$x_1$}] (x1) at (0, 0) {};
                \node [sdot, label = below: {$x_2$}] (x2) at (0.7, -1) {};
                \node [sdot, label = above: {$x_3$}] (x3) at (2, 0) {};
                \node [sdot, label = below: {$x_4$}] (x4) at ($(x2) + (x3) - (x1)$) {};
                \node [sdot, label = below: {$x_5$}] (x5) at ($(x1) + (x2) - (x3)$) {};
                \foreach \i\j\k\l\c\p in {x1/x2/x4/x3/MediumPurple3/0.1, x1/x3/x2/x5/DarkSlateGray3/0.1} {
                    \filldraw [\c, draw opacity = 0.25, fill opacity = \p] (\i.center) -- (\j.center) -- (\k.center) -- (\l.center) -- cycle;
                }
            \end{tikzpicture}
            \caption{The $k$-configuration from \ref{ex-item:collinearity}. We can see that $x_4$, $x_2$, and $x_5$ are forced to be collinear.}
        \end{figure}
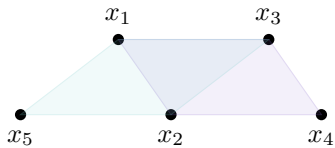
        \item \label{ex-item:cube} The $k$-configuration $\{x_1 - x_2 = x_3 - x_4 = x_5 - x_6 = x_7 - x_8, \, x_1 - x_3 = x_5 - x_7\}$ is valid and collinearity-free. However, it is $2$-heavy, since it implies $4$ independent equations on $8$ variables. 
        
        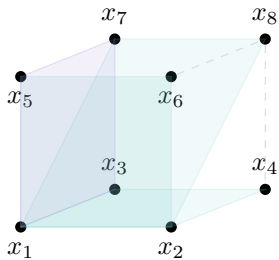
\begin{figure}[ht]
            \begin{tikzpicture}
                \node [sdot, label = below: {$x_1$}] (x1) at (0, 0) {};
                \node [sdot, label = below: {$x_2$}] (x2) at (2, 0) {};
                \node [sdot, label = above: {$x_3$}] (x3) at (1.25, 0.5) {};
                \node [sdot, label = above: {$x_4$}] (x4) at ($(x2) + (x3) - (x1)$) {};
                \node [sdot, label = below: {$x_5$}] (x5) at (0, 2) {};
                \foreach \i\j\k in {2/6/below, 3/7/above, 4/8/above} {
                    \node [sdot, label = \k: {$x_\j$}] (x\j) at ($(x\i) + (x5) - (x1)$) {};
                }
                \foreach \i\j\k\l\c\p in {x1/x2/x4/x3/DarkSlateGray3/0.1, x1/x2/x6/x5/DarkSlateGray3/0.15, x1/x2/x8/x7/DarkSlateGray3/0.1, x1/x3/x7/x5/MediumPurple3/0.1} {
                    \filldraw [\c, draw opacity = 0.25, fill opacity = \p] (\i.center) -- (\j.center) -- (\k.center) -- (\l.center) -- cycle;
                }
                \draw [gray!30, dashed] (x6) -- (x8) -- (x4);
            \end{tikzpicture}
            \caption{The $k$-configuration from \ref{ex-item:cube}, which describes an affine $3$-cube.}
        \end{figure}
        \item For any $p$, the star $\{x_1 + x_2 = x_3 + x_4 = \cdots = x_{2p - 1} + x_{2p}\}$ is $2$-good. One way to see this is that the content of any equation it implies is a linear combination of $x_1 + x_2$, $x_3 + x_4$, \ldots, $x_{2p - 1} + x_{2p}$ with coefficients summing to $0$. Any $t$ independent equations of this form must use at least $t + 1$ of $x_1 + x_2$, \ldots, $x_{2p - 1} + x_{2p}$, so they must contain at least $2(t + 1)$ variables. 
    \end{enumerate}
\end{example}

Now that we have these definitions, our goal is to prove the following two statements. We prove Lemma \ref{lem:random-constr} in Section \ref{sec:random-constr} (using a random construction), and Lemma \ref{lem:good-certify} in Sections \ref{sec:backbone}--\ref{sec:huge-star} (using the stability-type argument described in Subsection \ref{subsec:overview}). 

\begin{lemma} \label{lem:random-constr}
    Fix any $1 < c \leq 2$. For every $n \in \NN$, there is a set $A \subseteq \ZZ$ with $\abs{A} = n$ and $\abs{A - A} = O(n^c)$ such that for all distinct $a_1, \ldots, a_k \in A$, the $k$-configuration that they form is $c$-good. 
\end{lemma}

\begin{lemma} \label{lem:good-certify}
    Suppose that $c$ is sufficiently close to $2$ and that $k$ is even. Then every $c$-good $k$-configuration certifies at most $(k^2 - 2k)/4$ pairs. 
\end{lemma} 

(Note that equality holds for a star of size $k$, as seen in Example \ref{ex:star}.)

Together, Lemmas \ref{lem:random-constr} and \ref{lem:good-certify} immediately imply Theorem \ref{thm:main} --- if we take $A$ to be the set given by Lemma \ref{lem:random-constr}, then the number of distinct differences in any $k$-element subset $A' \subseteq A$ is $\binom{k}{2}$ minus the number of pairs certified by the $k$-configuration that the elements of $A'$ form. And Lemma \ref{lem:random-constr} guarantees that this $k$-configuration is $c$-good, so by Lemma \ref{lem:good-certify} this number of differences is at least $\binom{k}{2} - (k^2 - 2k)/4 = k^2/4$.  

\section{The random construction} \label{sec:random-constr}

In this section, we prove Lemma \ref{lem:random-constr} --- the statement that we can construct $n$-element sets $A$ with $O(n^c)$ differences which avoid all $c$-bad $k$-configurations --- using a random construction. (We say $A$ \emph{avoids} a $k$-configuration $\cC$ if there do not exist distinct $a_1, \ldots, a_k \in A$ such that $(a_1, \ldots, a_k)$ forms $\cC$.)

In order to avoid collinearity-inducing $k$-configurations, we need the following fact.

\begin{lemma} \label{lem:modified-behrend}
    Fix $\kappa > 0$. Then for all $n \in \NN$, there exists $S \subseteq \{1, \ldots, n\}$ of size $\abs{S} = n^{1 - o(1)}$ (where the asymptotic notation is as $\kappa$ is fixed and $n \to \infty$) such that for all nonzero $\alpha, \beta, \gamma \in \ZZ$ of magnitude at most $\kappa$, there do not exist distinct $s_1, s_2, s_3 \in S$ with $\alpha s_1 + \beta s_2 + \gamma s_3 = 0$. 
\end{lemma}

The construction for Lemma \ref{lem:modified-behrend} is a direct adaptation of the construction of $3$-AP-free sets by Behrend \cite{Beh46} (avoiding $3$-APs corresponds to taking $\kappa = 2$, but Behrend's construction actually works for any constant $\kappa$, with the same proof), so we defer it to Appendix \ref{app:modified-behrend}.

\begin{proof}[Proof of Lemma \ref{lem:random-constr}]
    Assume that $n$ is sufficiently large (with respect to $k$ and $c$); we will then construct $A$ with the desired properties such that $\abs{A} = n$ and $\abs{A - A} \leq n^c$. 

    We will first \emph{deterministically} avoid all collinearity-inducing $k$-configurations using Lemma \ref{lem:modified-behrend}. By definition, for every collinearity-inducing $k$-configuration $\cC$, there exist nonzero $\alpha, \beta, \gamma \in \QQ$ with $\alpha + \beta + \gamma = 0$ such that $\cC$ implies an equation of the form \[\alpha x_{i_1} + \beta x_{i_2} + \gamma x_{i_3} = 0.\] There are only finitely many $k$-configurations, so by clearing denominators and taking a maximum over all collinearity-inducing $k$-configurations, we can find a constant $\kappa$ (depending on $k$) such that for every collinearity-inducing $k$-configuration $\cC$, there exist such $\alpha, \beta, \gamma \in \ZZ$ with magnitude at most $\kappa$. 
    
    Now let $S$ be the set given by Lemma \ref{lem:modified-behrend} for this value of $\kappa$ and with $n$ replaced by $\sfloor{n^c}$. Then $S$ avoids all collinearity-inducing $k$-configurations, and we have $\abs{S} = n^{c - o(1)}$ and $\abs{S - S} \leq n^c$. 

    We will now use randomness to find some $A \subseteq S$ which avoids all $c$-heavy $k$-configurations. Let \[\rho = \frac{2n}{\abs{S}} = n^{1 - c + o(1)}.\] Define $B$ to be a $\rho$-random subset of $S$ (where we include each element of $S$ independently with probability $\rho$), so that $\EE[\abs{B}] = 2n$. Our goal is to delete a small number of elements from $B$ such that the resulting set avoids all $c$-heavy $k$-configurations. 
    
    For this, first consider a \emph{specific} $c$-heavy $k$-configuration $\cC$ which we wish to avoid. Since $\cC$ is $c$-heavy, it implies some collection $\cT$ of $t$ independent equations on $v$ variables, for some $v < ct + 1$. Then it suffices to look at all solutions to $\cT$ in $B$ with distinct entries (looking at only the $v$ variables present in $\cT$) and delete one element from each --- if $b_1, \ldots, b_k \in B$ are distinct and form $\cC$, then the $v$ numbers $b_i$ corresponding to variables $x_i$ in $\cT$ must form a solution to $\cT$. 

    To bound the expected number of elements this causes us to delete, first note that the number of solutions to $\cT$ in $S$ is at most $\abs{S}^{v - t}$. One way to see this is to imagine putting $\cT$ into row echelon form; then we have $v - t$ free variables, each of whose values can be chosen in $\abs{S}$ ways, and the values of the remaining variables are determined in terms of the values of the free variables.

    Meanwhile, for any such solution with distinct entries, the probability that all $v$ of its entries are placed into $B$ is $\rho^v$. So the expected number of solutions to $\cT$ in $B$ with distinct entries is at most \[\rho^v \abs{S}^{v - t} = n^{v(1 - c) + o(1)} \cdot n^{c(v - t) - o(1)} = n^{v - ct + o(1)} = o(n).\] (Note that $v$ and $c$ are at most $k$ and $v - ct < 1$, so we can bound $v - ct$ away from $1$ by a constant only depending on $k$.) 

    Now let $B'$ be the set obtained by taking $B$ and performing such deletions for \emph{every} $c$-heavy $k$-configuration $\cC$. Each $c$-heavy $k$-configuration $\cC$ causes us to delete $o(n)$ elements in expectation, and the total number of $k$-configurations is a constant only depending on $k$, so the \emph{total} number of elements we delete is also $o(n)$ in expectation; this means \[\EE[\abs{B'}] = \EE[\abs{B}] - o(n) = 2n - o(n) \geq n\] (assuming that $n$ is sufficiently large). Finally, this means there is some outcome of the randomness under which $\abs{B'} \geq n$, and we can obtain $A$ from $B'$ by deleting elements arbitrarily until it has size exactly $n$. 
\end{proof}

\section{Backbone of the proof of Lemma \ref{lem:good-certify}} \label{sec:backbone}

In this section, we present the high-level argument used to prove Lemma \ref{lem:good-certify} --- the statement that no $c$-good $k$-configuration certifies more pairs than a star of size $k$ does --- following the outline in Subsection \ref{subsec:overview}. Some steps of this argument require lemmas whose proofs are fairly technical; we will defer these proofs to Sections \ref{sec:observations}--\ref{sec:huge-star}.

Throughout the argument, we will have two parameters $\eps$ and $k_0$, in addition to $c$. We think of these parameters in the following way. 
\begin{itemize}
    \item We use $\eps$ to denote the error parameter for our stability argument; we think of it as a sufficiently small absolute constant. 
    \item For most of the argument, we will assume that $k$ is sufficiently large with respect to $\eps$. (Step \ref{step:baseline} of the outline in Subsection \ref{subsec:overview} will be enough to finish the proof for small $k$.) We use $k_0$ to indicate what `sufficiently large' means; so $k_0$ is a large constant depending on $\eps$.
    \item We choose $c$ such that $2 - c$ is small with respect to $\eps$ and $k_0$. 
\end{itemize}
Setting $\eps \leq 1/4096$, $k_0 \geq 32/\eps^2$, and $c \geq 2 - \min\{\eps^2/32, 2/k_0\}$ works; this allows us to get $c = 2 - 2^{-29}$.

\subsection{The technical lemmas} \label{subsec:technical}

In this subsection, we collect the statements of the lemmas whose proofs we defer. First, for Step \ref{step:baseline} of the outline (the baseline bound), we need the following lemma. 

\begin{lemma} \label{lem:baseline-technical}
    Fix $i \in [k]$, and let $\cS$ be a $c$-good collection of independent difference equalities on $x_1$, \ldots, $x_i$ which all contain $x_i$, with $\abs{\cS} \leq s$. Then $\cS$ certifies at most $2s$ pairs $(i, \bullet)$. 
\end{lemma}

When we say that $\cS$ certifies at most $2s$ pairs $(i, \bullet)$, we mean there are at most $2s$ indices $j < i$ for which $\cS$ certifies $(i, j)$. 

The intuition behind Lemma \ref{lem:baseline-technical} is that a \emph{single} difference equality containing $x_i$ certifies two pairs $(i, \bullet)$ --- for example, $x_i - x_1 = x_2 - x_3$ certifies $(i, 1)$ and $(i, 2)$. So Lemma \ref{lem:baseline-technical} would be immediate if $\cS$ did not imply any difference equalities containing $x_i$ other than the ones already in $\cS$. The difficulty comes from the fact that $\cS$ \emph{can} imply additional difference equalities, which could potentially certify extra pairs $(i, \bullet)$. However, it turns out that when this happens, some of the difference equalities in $\cS$ must certify the \emph{same} pairs $(i, \bullet)$, and this overlap compensates for those extra pairs. 

For Step \ref{step:stability} of the outline (the stability argument), we need the following two lemmas. 

\begin{lemma} \label{lem:contained-in-s}
    Fix $i \in [k]$ and $s \geq 1/\eps$, and let $\cS$ be a $c$-good collection of independent difference equalities on $x_1$, \ldots, $x_i$ which all contain $x_i$, with $\abs{\cS} \leq s$. Suppose that $\cS$ certifies at least $2(1 - \eps)s$ pairs $(i, \bullet)$ and does not imply any star of size at least $2(1 - 16\eps)s$. Then $\cS$ certifies at most $(1 - 2\eps)s^2$ pairs $(i', j') \in [i - 1]^2$. 
\end{lemma}

The intuition behind Lemma \ref{lem:contained-in-s} is that if $\cS$ consisted of $s$ equations defining a star of size $2(s + 1)$ (e.g., $\cS = \{x_i + x_1 = x_2 + x_3, \, x_i + x_1 = x_4 + x_5, \, \ldots, \, x_i + x_1 = x_{2s} + x_{2s + 1}\}$), then $\cS$ would certify exactly $2s$ pairs $(i, \bullet)$ and exactly $s(s - 1)$ pairs $(i', j') \in [i - 1]^2$. Lemma \ref{lem:baseline-technical} states that $2s$ is the maximum possible number of pairs $(i, \bullet)$ that \emph{any} collection $\cS$ of size at most $s$ could certify; and Lemma \ref{lem:contained-in-s} states that if $\cS$ is close to this maximum, then either it must be similar to a star of size $2(s + 1)$ (in that it must imply a star of nearly this size), or the number of pairs $(i', j') \in [i - 1]^2$ that it certifies must be substantially smaller. 

\begin{lemma} \label{lem:intersect-r}
    Fix $i \in [k]$, and let $\cS$ be a collection of difference equalities containing $x_i$ and $\cR$ a collection of difference equalities not containing $x_i$ such that the following conditions hold.
    \begin{itemize}
        \item $\cR \cup \cS$ is $c$-good and linearly independent. 
        \item Every difference equality containing $x_i$ that is implied by $\cR \cup \cS$ is actually implied by $\cS$ alone. 
    \end{itemize}
    Then the number of pairs certified by $\cR \cup \cS$ but not $\cS$ is at most $2048\abs{\cR}^2$.
\end{lemma}

When we apply Lemma \ref{lem:intersect-r}, $\cS$ will be as in Lemma \ref{lem:contained-in-s}, and $\cR$ will be much smaller than $\cS$ (i.e., we will have $\abs{\cS} \leq s$ and $\abs{\cR} \leq \eps s$). Lemma \ref{lem:contained-in-s} gives us a bound slightly smaller than $s^2$ for the number of pairs certified by $\cS$ alone; the intuition behind Lemma \ref{lem:intersect-r} is that it means taking $\cR$ into account will have very little effect on this bound (it will only give us $2048\eps^2 s^2$ extra pairs, which is much smaller than $s^2$). 

Finally, for Step \ref{step:huge-star}, we need the following lemma. 

\begin{lemma} \label{lem:outside-star}
    Let $\cP = \{x_1 + x_2 = x_3 + x_4, \, x_1 + x_2 = x_5 + x_6, \, \ldots, \, x_1 + x_2 = x_{2p - 1} + x_{2p}\}$, and fix $i > 2p$. Let $\cS$ be a collection of difference equalities containing $x_i$ such that $\cS \cup \cP$ is independent and $c$-good and $\cS \cup \cP$ does not imply $x_1 + x_2 = x_i + x_j$ for any $j > 2p$. Then $\cS \cup \cP$ certifies at most $6\abs{\cS}$ pairs $(i, \bullet)$. 
\end{lemma}

The intuition behind Lemma \ref{lem:outside-star} is that it describes what happens when we start with a huge star and try to add a small number of equations to it (when we apply Lemma \ref{lem:outside-star}, $\cS$ will be very small). Specifically, it states that if we have a variable $x_i$ and try to introduce it into our picture without making it part of the huge star, then we can only certify very few pairs $(i, \bullet)$. 

In the remainder of this section, we will prove Lemma \ref{lem:good-certify} assuming these lemmas. In Section \ref{sec:observations}, we will present some preliminary observations which will be useful for the proofs of several of these lemmas. We will then prove Lemmas \ref{lem:baseline-technical} and \ref{lem:contained-in-s} in Section \ref{sec:impls-at-i}, Lemma \ref{lem:intersect-r} in Section \ref{sec:intersect-r}, and Lemma \ref{lem:outside-star} in Section \ref{sec:huge-star}. 

\subsection{A baseline bound}

As described in Step \ref{step:baseline}, the first step of the argument is to prove a slightly weaker bound on the number of pairs that $\cC$ certifies. 

\begin{lemma} \label{lem:baseline}
    If the solution space of $\cC$ has dimension $d$, then $\cC$ certifies at most $(k - d)(k - d + 1)$ pairs. 
\end{lemma}

\begin{proof}
    First, imagine that we put $\cC$ into row echelon form. This splits the variables $x_1$, \ldots, $x_k$ into sets of free and non-free variables, such that in order to construct a solution to $\cC$, we can choose the values of the free variables arbitrarily, and these values uniquely determine the values of the non-free variables. In particular, the number of free variables must be exactly $d$; so by renaming variables, we can assume without loss of generality that the free variables are precisely $x_1$, \ldots, $x_d$. 

    Then $\cC$ cannot imply any difference equalities on $x_1$, \ldots, $x_d$, so for each $i \leq d$, the number of pairs $(i, \bullet)$ certified by $\cC$ is simply $0$. 

    Meanwhile, for each $i > d$, we can bound the number of pairs $(i, \bullet)$ certified by $\cC$ as follows: Fix $i$, and let $\cS$ be a maximal linearly independent set of difference equalities on $x_1$, \ldots, $x_i$ involving $x_i$ that are implied by $\cC$. This means every difference equality on $x_1$, \ldots, $x_i$ involving $x_i$ that is implied by $\cC$ is in fact implied by $\cS$ (otherwise we could add it to $\cS$ while preserving independence); so every pair $(i, \bullet)$ certified by $\cC$ is in fact certified by $\cS$. But we must have $\abs{\cS} \leq i - d$ --- this is because $x_1$, \ldots, $x_d$ are all free variables, so $\cC$ cannot imply more than $i - d$ linearly independent equations on $x_1$, \ldots, $x_i$. Then Lemma \ref{lem:baseline-technical} means that $\cS$ (and therefore $\cC$) certifies at most $2(i - d)$ pairs $(i, \bullet)$. 

    Finally, summing over all $i$, the total number of pairs that $\cC$ certifies is at most \[2 \cdot 1 + 2 \cdot 2 + \cdots + 2 \cdot (k - d) = (k - d)(k - d + 1).\qedhere\] 
\end{proof}

Recall that our ultimate goal is to prove that $\cC$ certifies at most $(k^2 - 2k)/4$ pairs (this is the number of pairs certified by a star of size $k$; a star of size $k$ achieves equality in Lemma \ref{lem:baseline} with $d = k/2 - 1$).

To see how Lemma \ref{lem:baseline} compares to the desired bound, note that if the dimension of the solution space of $\cC$ is $d$, then $\cC$ implies some collection of $k - d$ linearly independent equations (which contain at most $k$ variables); since $\cC$ is $c$-good, this means \[k \geq c(k - d) + 1,\] which rearranges to $k - d \leq (k - 1)/c$. If $k$ is small with respect to $2 - c$, then we have $(k - 1)/c < k/2$, so $k - d \leq k/2 - 1$. Then Lemma \ref{lem:baseline} immediately gives the desired bound, and we are done. 

So for the remainder of the argument, we can assume that $k$ is large (more precisely, that $k \geq k_0$ --- we defined parameters so that $2 - c$ is small with respect to $k_0$). Then Lemma \ref{lem:baseline} gives a bound of \[(k - d)(k - d + 1) \leq \floor{\frac{k - 1}{c}}\left(\floor{\frac{k - 1}{c}} + 1\right) \leq \frac{k^2}{c^2}.\] This differs from the desired bound of $(k^2 - 2k)/4$ by a tiny constant factor (which we can make arbitrarily small by making $c$ sufficiently close to $2$). 

\subsection{A stability statement}

The next step of the argument, as described in Step \ref{step:stability}, is to prove a stability-type statement --- that if $\cC$ is even \emph{close} to the weaker bound of $k^2/c^2$ that we get from Lemma \ref{lem:baseline}, then it must be similar to a star of size $k$ (in that it must imply a star of size nearly $k$). 

\begin{lemma} \label{lem:stability}
    Suppose that $\cC$ does not imply any star of size at least $(1 - 17\eps)k$. Then the number of pairs that it certifies is at most \[\left(1 - \frac{\eps^2}{8}\right) \cdot \frac{k^2}{c^2}.\] 
\end{lemma}

\begin{proof}
    Let $d$ be the dimension of the solution space of $\cC$ (which satisfies $d \leq (k - 1)/c$). As in the proof of Lemma \ref{lem:baseline}, imagine that we put $\cC$ into row echelon form and rename variables so that $x_1$, \ldots, $x_d$ are free (i.e., in order to construct a solution to $\cC$, we can choose the values of $x_1$, \ldots, $x_d$ arbitrarily, and these determine the values of the remaining variables). Then $\cC$ cannot certify any pairs $(i, \bullet)$ with $i \leq d$, and it certifies at most $2(i - d)$ pairs $(i, \bullet)$ for each $i > d$ (by the same argument as in the proof of Lemma \ref{lem:baseline}). We say an index $i > d$ is \emph{near-saturated} if $\cC$ certifies at least $2(1 - \eps)(i - d)$ pairs $(i, \bullet)$. 

    \case{1}{The largest near-saturated index $i$ satisfies $i - d \geq (1 - \eps)k/2$} Here, the main idea is that we will use Lemmas \ref{lem:contained-in-s} and \ref{lem:intersect-r} to save a $(1 - \eps)$-factor on the number of pairs $(i', j') \in [i]^2$ that $\cC$ certifies, compared to the bound of $2 \cdot 1 + \cdots + 2 \cdot (i - d) = (i - d)(i - d + 1)$ from Lemma \ref{lem:baseline}. Meanwhile, we will use the fact that no indices past $i$ are near-saturated to individually save a $(1 - \eps)$-factor on the number of pairs $(i', \bullet)$ that $\cC$ certifies for each $i' > i$. Combining these will allow us to save a $(1 - \eps)$-factor compared to the bound from Lemma \ref{lem:baseline}. 

    First let $\cS$ be a maximal collection of independent difference equalities on $x_1$, \ldots, $x_i$ involving $x_i$ that are implied by $\cC$. Then let $\cR$ be a maximal collection of difference equalities on $x_1$, \ldots, $x_{i - 1}$ implied by $\cC$ such that $\cR \cup \cS$ is independent.
    \begin{itemize}
        \item The maximality of $\cS$ means that every difference equality on $x_1$, \ldots, $x_i$ involving $x_i$ that is implied by $\cC$ is in fact implied by $\cS$ (in particular, all pairs $(i, \bullet)$ certified by $\cC$ are actually certified by $\cS$).
        \item The maximality of $\cR$ means that every difference equality on $x_1$, \ldots, $x_{i - 1}$ that is implied by $\cC$ is implied by $\cR \cup \cS$ (in particular, all pairs $(i', j') \in [i - 1]^2$ certified by $\cC$ are certified by $\cR \cup \cS$). 
    \end{itemize}
    
    First, since $x_1$, \ldots, $x_d$ are free variables, $\cC$ can imply at most $i - d$ independent equations on $x_1$, \ldots, $x_i$. So letting $s = i - d$, we have $\abs{\cR} + \abs{\cS} \leq s$. Furthermore, since $i$ is near-saturated, $\cS$ certifies at least $2(1 - \eps)s$ pairs $(i, \bullet)$, so Lemma \ref{lem:baseline-technical} means that $\abs{\cS} \geq (1 - \eps)s$, and therefore $\abs{\cR} \leq \eps s$. 

    Then $\cS$ satisfies the conditions for Lemma \ref{lem:contained-in-s} --- the assumption that $k \geq k_0$ (where $k_0$ is large with respect to $\eps$) means that $s \geq 1/\eps$, and the assumption that $\cC$ does not imply a star of size at least $(1 - 17\eps)k$ means that $\cS$ does not imply a star of size at least $2(1 - 16\eps)s$, since \[2(1 - 16\eps)s \geq 2(1 - 16\eps) \cdot \frac{(1 - \eps)k}{2} \geq (1 - 17\eps)k.\] So Lemma \ref{lem:contained-in-s} means that $\cS$ certifies at most $(1 - 2\eps)s^2$ pairs $(i', j') \in [i - 1]^2$. Meanwhile, $\cR$ and $\cS$ satisfy the conditions for Lemma \ref{lem:intersect-r}, which gives that $\cR \cup \cS$ certifies at most \[2048\abs{\cR}^2 \leq 2048\eps^2s^2\] additional pairs $(i', j') \in [i - 1]^2$ (i.e., ones not already certified by $\cS$). Finally, $\cS$ certifies at most $2s$ pairs $(i, \bullet)$. So in total, the number of pairs $(i', j') \in [i]^2$ certified by $\cR \cup \cS$ (and therefore by $\cC$) is at most \[(1 - 2\eps)s^2 + 2048\eps^2s^2 + 2s \leq (1 - \eps)s(s + 1)\] (this inequality again uses the fact that $\eps$ is small and $s$ is large with respect to $\eps$). Recall that $s = i - d$; so we have indeed saved a $(1 - \eps)$-factor on the number of pairs $(i', j') \in [i]^2$ that $\cC$ certifies compared to the bound of $(i - d)(i - d + 1)$ from Lemma \ref{lem:baseline}. 

    Finally, for each $i' > d$, the fact that $i'$ is not near-saturated means that $\cC$ certifies at most $2(1 - \eps)(i' - d)$ pairs $(i', \bullet)$. So in total, the number of pairs that $\cC$ certifies is at most \[(1 - \eps)(i - d)(i - d + 1) + 2(1 - \eps)((i - d + 1) + \cdots + (k - d)) = (1 - \eps)(k - d)(k - d + 1).\] Since $k - d \leq (k - 1)/c$, this is at most $(1 - \eps)k^2/c^2$, which is smaller than the desired bound.
    
    \case{2}{No indices $i$ with $i - d \geq (1 - \eps)k/2$ are near-saturated} Here, the idea is that we can save a substantial amount on the number of pairs $(i, \bullet)$ that $\cC$ certifies for every index $i$ with $i - d \geq (1 - \eps)k/2$, compared to the bound of $2(i - d)$ from Lemma \ref{lem:baseline}. And there are reasonably many such indices, so this will give a substantial saving on the total number of pairs that $\cC$ certifies. 

    To formalize this, imagine that we compare the two sums 
    \begin{equation*}
        (\text{I}) = \sum_{i > d} \#(\text{pairs $(i, \bullet)$ certified by $\cC$}) \quad \text{and} \quad (\text{II}) = 2 \cdot 1 + 2 \cdot 2 + \cdots + 2 \cdot \floor{\frac{k - 1}{c}}.
    \end{equation*}
    The first is the number of pairs actually certified by $\cC$, and the second is the bound we get from Lemma \ref{lem:baseline} after using the fact that $k - d \leq (k - 1)/c$. (In particular, we have $(\text{II}) \leq k^2/c^2$.)
    
    We will compare $(\text{I})$ and $(\text{II})$ termwise. (It is possible that $(\text{II})$ has more terms than $(\text{I})$; in that case, we imagine appending $0$'s to the end of $(\text{I})$.) Every term in $(\text{I})$ is at most the corresponding term in $(\text{II})$. For all $(1 - \eps)k/2 \leq j \leq (k - 1)/c$, the $j$th term in $(\text{II})$ is $2j$, while the $j$th term in $(\text{I})$ is at most \[2(1 - \eps)j \leq 2j - \frac{\eps k}{2}.\] (If $j + d \leq k$, then this follows from the assumption that $i = j + d$ is not near-saturated; otherwise, the $j$th term in $(\text{I})$ is $0$.) The number of such $j$ is at least \[\frac{k - 1}{c} - \frac{(1 - \eps)k}{2} - 1 \geq \frac{\eps k}{4}\] (since $c \leq 2$ and we assumed $k \geq k_0$, i.e., that $k$ is large with respect to $\eps$). So we get that \[(\text{I}) \leq (\text{II}) - \frac{\eps k}{2} \cdot \frac{\eps k}{4} \leq \frac{k^2}{c^2} - \frac{\eps^2 k^2}{8} \leq \left(1 - \frac{\eps^2}{8}\right)\frac{k^2}{c^2}.\qedhere\] 
\end{proof}

In order to prove Lemma \ref{lem:good-certify}, we want to show that $\cC$ certifies at most $(k^2 - 2k)/4$ pairs. Since we chose $2 - c$ to be small (and assumed $k$ is large) with respect to $\eps$, we have \[\left(1 - \frac{\eps^2}{8}\right) \cdot \frac{k^2}{c^2} \leq \frac{k^2 - 2k}{4}.\] So if $\cC$ does not imply a star of size at least $(1 - 17\eps)k$, then we are done; it remains to consider the case where $\cC$ does implies a star of this size. 

\subsection{The huge-star case}

Finally, we handle the case where $\cC$ implies a huge star, corresponding to Step \ref{step:huge-star} of the outline. The proof in this case has two components. The first is that $\cC$ certifies very few pairs $(i, \bullet)$ for variables $x_i$ outside the star; this is captured by Lemma \ref{lem:outside-star}. The second is that $\cC$ cannot certify any extra pairs \emph{within} the star; this is captured by the following lemma. 

\begin{lemma} \label{lem:within-star}
    Suppose that $\cC$ implies a star $\{x_1 + x_2 = x_3 + x_4 = \cdots = x_{2p - 1} + x_{2p}\}$. Then $\cC$ cannot imply any difference equality on $x_1$, \ldots, $x_{2p}$ other than the ones implied by this star. 
\end{lemma}

\begin{proof}
    We call the variables $x_{2i - 1}$ and $x_{2i}$ \emph{opposites} in our star (for $1 \leq i \leq p$). Assume for contradiction that $\cC$ does imply a difference equality $(*)$ on $x_1$, \ldots, $x_{2p}$ which is not implied by our star, and consider the set of variables consisting of the four variables in $(*)$ and their opposites. This set contains $2t$ variables for some $2 \leq t \leq 4$ (where $t$ depends on how many pairs of opposites appear in $(*)$). 

    But the star implies $t - 1$ independent equations on these $2t$ variables, namely the equations stating that each pair of opposites has equal sum. (For example, if $(*)$ is the equation $x_1 - x_5 = x_9 - x_7$, then we would take the $8$ variables $x_1$, $x_2$, $x_5$, $x_6$, $x_7$, $x_8$, $x_9$, $x_{10}$ and the $3$ equations $x_1 + x_2 = x_5 + x_6$, $x_1 + x_2 = x_7 + x_8$, and $x_1 + x_2 = x_9 + x_{10}$.) And $(*)$ provides another independent equation on these variables, so this means $\cC$ implies $t$ independent equations on $2t$ variables. Since $t \leq 4$ and $c$ is close to $2$, we have $2t < ct + 1$, so this contradicts the fact that $\cC$ is $c$-good. 
\end{proof}

By combining Lemmas \ref{lem:outside-star} and \ref{lem:within-star}, we obtain the following bound for the huge-star case. 

\begin{lemma} \label{lem:huge-star}
    Suppose that the largest star implied by $\cC$ has size $2p$ for some $(1 - 17\eps)k \leq 2p \leq k$. Then $\cC$ certifies at most $p(k/2 - 1)$ pairs. 
\end{lemma}

\begin{proof}
    By renaming variables, we can assume that the largest star implied by $\cC$ is \[\{x_1 + x_2 = x_3 + x_4 = \cdots = x_{2p - 1} + x_{2p}\}.\] Then by Lemma \ref{lem:within-star}, the number of pairs $(i, j) \in [2p]^2$ certified by $\cC$ is exactly the number of pairs certified by this star, which is $p^2 - p$ (as computed in Example \ref{ex:star}). 
    
    Meanwhile, for each $i > 2p$, we bound the number of pairs $(i, \bullet)$ that $\cC$ certifies using Lemma \ref{lem:outside-star}. Let \[\cP = \{x_1 + x_2 = x_3 + x_4, \, x_1 + x_2 = x_5 + x_6, \, \ldots, x_1 + x_2 = x_{2p - 1} + x_{2p}\}\] be as in the statement of Lemma \ref{lem:outside-star}, and let $\cS$ be a maximal collection of difference equalities involving $x_i$ implied by $\cC$ such that $\cS \cup \cP$ is linearly independent. The maximality of $\cS$ means that every pair $(i, \bullet)$ certified by $\cC$ is certified by $\cS \cup \cP$. Meanwhile, since $\cC$ is $c$-good and $\cS \cup \cP$ is linearly independent and contains at most $k$ variables, we have \[\abs{\cS} + \abs{\cP} \leq \frac{k - 1}{c},\] and since $\abs{\cP} = p - 1 \geq (1 - 17\eps)k/2 - 1$, this means \[\abs{\cS} \leq \frac{k - 1}{c} - \frac{(1 - 17\eps)k}{2} + 1 \leq 10\eps k \leq 24\eps p\] (since $2 - c$ is small with respect to $\eps$, and we assumed $k$ is large with respect to $\eps$). 
    
    Then Lemma \ref{lem:outside-star} gives that $\cS \cup \cP$ (and therefore $\cC$) certifies at most $144 \eps p$ pairs $(i, \bullet)$. (The fact that our star is the largest one implied by $\cC$ means that the condition for Lemma \ref{lem:outside-star} --- that $x_i$ cannot be added to the star --- is satisfied.) 
    
    Finally, there are $k - 2p$ indices $i > 2p$, and $\cC$ certifies at most $144\eps p \leq p/2$ pairs $(i, \bullet)$ for each such $i$. So the total number of pairs it certifies is at most \[p^2 - p + (k - 2p) \cdot \frac{p}{2} = p\left(\frac{k}{2} - 1\right).\qedhere\] 
\end{proof}

Lemma \ref{lem:huge-star} completes the proof of Lemma \ref{lem:good-certify} for the case where $\cC$ implies a huge star (we always have $p \leq k/2$, so $p(k/2 - 1) \leq (k^2 - 2k)/4$); and as discussed earlier, Lemma \ref{lem:stability} handles the case where $\cC$ does not imply a huge star. So this completes the proof of Lemma \ref{lem:good-certify} (modulo the technical lemmas stated in Subsection \ref{subsec:technical}, which we will prove in the following sections). 

\section{Some useful observations} \label{sec:observations}

It now remains to prove the technical lemmas stated in Subsection \ref{subsec:technical}, which we will do in Sections \ref{sec:impls-at-i}--\ref{sec:huge-star}. In this section, we present a few preliminary observations that will be useful for several of those proofs. (Lemma \ref{lem:2-good-min-impl} and Claim \ref{claim:2-eqns-5-vars} are taken from \cite{Das23}, but we restate their proofs here to keep this paper self-contained, and because the ideas used to prove Lemma \ref{lem:2-good-min-impl} are also useful in other arguments.)

We say a \emph{minimal implication} is a collection of independent difference equalities $\{(*_1), \ldots, (*_t)\}$ which minimally implies a difference equality $(*)$, and we say $(*)$ is \emph{produced} by this minimal implication. In all our technical lemmas, we start with a collection of difference equalities with certain properties, and we want to understand the pairs that this collection certifies; this means we want to understand what other difference equalities it implies. We will typically do so by looking at minimal implications, and we will often use the results of this section to make claims about the structure of those minimal implications. 

The first observation is that if we have a \emph{constant-sized} collection of difference equalities which we know is $c$-good, then we can actually say it is $2$-good. 

\begin{claim} \label{claim:c-to-2}
    If $\cT$ is a set of difference equalities with $\abs{\cT} < 1/(2 - c)$ and $\cT$ is $c$-good, then $\cT$ is $2$-good. 
\end{claim}

\begin{proof}
    By definition, $\cT$ is $c$-light if and only if every $t$ linearly independent equations implied by $\cT$ together contain at least $ct + 1$ variables. But if $\cT$ implies $t$ linearly independent equations, we must have $t \leq \abs{\cT}$. And for all such values of $t$, we have $ct + 1 > 2t$, which means that requiring a collection of $t$ equations to contain at least $ct + 1$ variables is equivalent to requiring it to contain at least $2t + 1$ variables. 

    So if $\cT$ is $c$-light then it is also $2$-light; and the remaining conditions in the definition of being $c$-good (being valid and collinearity-free) do not depend on $c$. 
\end{proof}

The reason Claim \ref{claim:c-to-2} is useful because it turns out that we can say a lot about the structure of $2$-good minimal implications, as seen in the following lemma. 

\begin{lemma} \label{lem:2-good-min-impl}
    Let $\cT = \{(*_1), \ldots, (*_t)\}$ be a $2$-good minimal implication producing a difference equality $(*)$. Then the following statements hold. 
    \begin{enumerate}[(i)]
        \item \label{item:variable-counts} Either $\cT$ contains $2t + 2$ variables and each appears twice among $(*_1)$, \ldots, $(*_t)$, $(*)$; or $\cT$ contains $2t + 1$ variables, one appears four times, and the others each appear twice. 
        \item \label{item:signs-pm1} We have ${*} = \pm {*_1} \pm {*_2} \pm \cdots \pm {*_t}$ (for some choice of signs).
        \item \label{item:unique} $\cT$ cannot produce any difference equality other than $(*)$. 
    \end{enumerate}
\end{lemma} 

\begin{proof}
    Since $\cT$ minimally implies $(*)$, we can write 
    \begin{equation}
        {*} = c_1{*_1} + \cdots + c_t{*_t} \label{eqn:min-impl}
    \end{equation}
    for some nonzero $c_1, \ldots, c_t \in \QQ$. This means each variable which appears in $\cT$ must appear in at least two of $(*_1)$, \ldots, $(*_t)$, $(*)$. Furthermore, each equation has four variables, so the total number of slots for variables to appear in (i.e., the sum of the number of times each variable appears) is exactly $4(t + 1)$. But since $\cT$ is $2$-good, the number of variables appearing in $\cT$ must be at least $2t + 1$. There are only three ways in which this can occur: 
    \begin{enumerate}[(a)]
        \item There are $2t + 2$ variables, and each appears in exactly two equations. 
        \item There are $2t + 1$ variables, one variable appears in four equations, and the others appear in two. 
        \item \label{item:3-3} There are $2t + 1$ variables, two variables appear in three equations, and the others appear in two. 
    \end{enumerate}
    In other words, we have nearly proved \ref{item:variable-counts}, except that we have an extra case \ref{item:3-3}. We will use this weaker version of \ref{item:variable-counts} to deduce \ref{item:signs-pm1} (i.e., that $c_1, \ldots, c_t \in \{\pm 1\}$), and then use \ref{item:signs-pm1} to eliminate this extra case and complete the proof of \ref{item:variable-counts} (as well as to prove \ref{item:unique}). 

    We call a variable \emph{ordinary} if it appears in exactly two of $(*_1)$, \ldots, $(*_t)$, $(*)$; so far, we have shown that at most two variables in $\cT$ are not ordinary. Note that if $x_i$ is ordinary, the two equations that $x_i$ appears in must have coefficients of the same magnitude in \eqref{eqn:min-impl}. So we create a graph on $(*_1)$, \ldots, $(*_t)$, $(*)$ where for each ordinary variable $x_i$, we draw an edge between the two equations it appears in; to prove \ref{item:signs-pm1}, it suffices to show that this graph is connected. 

    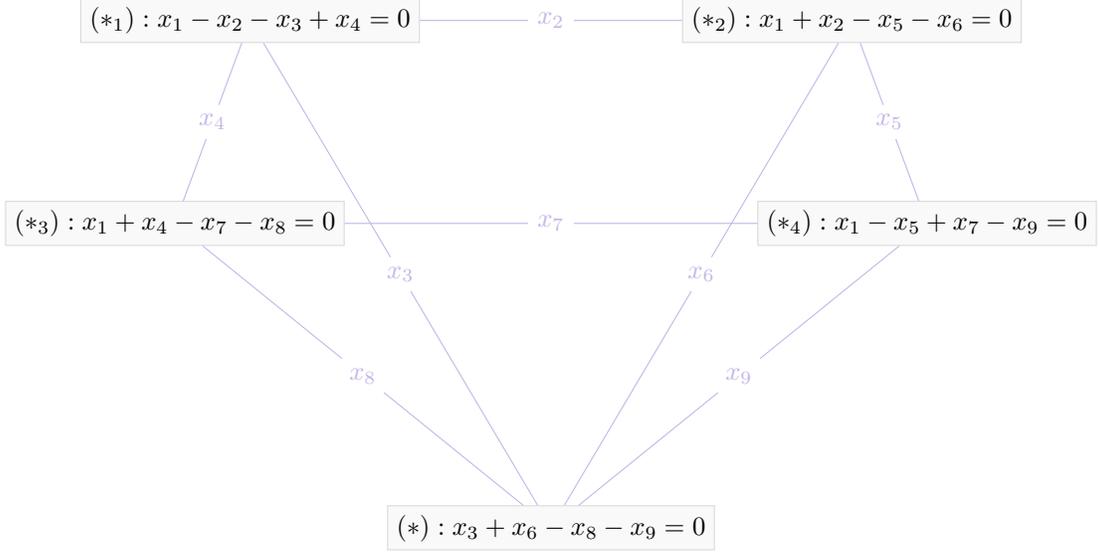
\begin{figure}[ht]
        \centering
        \begin{tikzpicture}[yscale = 0.9]
            \begin{scope}[every node/.style = {draw = gray!30, fill = gray!5}]
                \node (s1) at (-4, 3) {$(*_1) : x_1 - x_2 - x_3 + x_4 = 0$};
                \node (s2) at (4, 3) {$(*_2) : x_1 + x_2 - x_5 - x_6 = 0$};
                \node (s3) at (-5, 0) {$(*_3) : x_1 + x_4 - x_7 - x_8 = 0$};
                \node (s4) at (5, 0) {$(*_4) : x_1 - x_5 + x_7 - x_9 = 0$};
                \node (s) at (0, -4.5) {$(*) : x_3 + x_6 - x_8 - x_9 = 0$};
            \end{scope}
            \foreach \i\j\k\l in {s1/s2/2, s1/s3/4, s1/s/3, s2/s4/5, s2/s/6, s3/s/8, s4/s/9, s3/s4/7} {
                \node [MediumPurple3!50] (x) at ($(\i)!0.5!(\j)$) {$x_\k$};
                \draw [MediumPurple3!50] (\i) -- (x) -- (\j);
            }
        \end{tikzpicture}
        \caption{An example of a minimal implication (we have ${*} = -{*_1} - {*_2} + {*_3} + {*_4}$) and the corresponding graph used to prove \ref{item:signs-pm1}, where every edge is labelled with the ordinary variable it corresponds to.}
    \end{figure}

    Assume for contradiction that this graph is \emph{not} connected, and consider some connected component that does not contain $(*)$; without loss of generality, suppose that this connected component consists of $(*_1)$, \ldots, $(*_{t'})$ for some $t' \leq t$. Now consider the equation $(*')$ defined by \[{*'} = c_1{*_1} + \cdots + c_t{*_{t'}}.\] No ordinary variable $x_i$ can appear in $(*')$ --- if the right-hand side contains one of the two equations where $x_i$ appears, then it must also contain the other one (by the way we defined our graph), and these two appearances must cancel each other out (because $x_i$ cancels out of \eqref{eqn:min-impl}). Since at most two variables are not ordinary, this means $(*')$ contains at most two variables. 

    But $(*')$ cannot be identically zero because $\cT$ is linearly independent. And its coefficients must sum to $0$, so it must be of the form $\alpha x_i - \alpha x_j = 0$ for some $\alpha \neq 0$ and $i \neq j$; this contradicts the validity of $\cT$. 

    So we have shown that our graph is connected, which completes the proof of \ref{item:signs-pm1}.
    
    Now \ref{item:signs-pm1} means that every variable appears an even number of times in $(*_1)$, \ldots, $(*_t)$, $(*)$. This completes the proof of \ref{item:variable-counts} by ruling out the extra case \ref{item:3-3}. It also proves \ref{item:unique} --- it shows that two difference equalities minimally implied by $\cT$ would have to contain the same set of variables (namely, the ones appearing an odd number of times among $(*_1)$, \ldots, $(*_t)$), and $\cT$ cannot imply two distinct difference equalities on the same four variables (as seen in Example \ref{ex:c-good}\ref{ex-item:2-eqns-4-vars}, this would contradict its validity). 
\end{proof}

We say a collection of $t$ linearly independent difference equalities is \emph{$2$-full} if it contains exactly $2t + 1$ variables. In our proofs of the technical lemmas, most $2$-good minimal implications we encounter will be $2$-full (i.e., they will fall into the second case of Lemma \ref{lem:2-good-min-impl}\ref{item:variable-counts}). So we now prove a few statements that help us handle how such objects interact with each other. 

\begin{lemma} \label{lem:2-full-intersection}
    Suppose that $\cT_1$ and $\cT_2$ are non-disjoint sets of difference equalities whose union is linearly independent and $2$-good. If $\cT_1$ and $\cT_2$ are both $2$-full, then so is $\cT_1 \cap \cT_2$. 
\end{lemma}

\begin{proof}
    Suppose that $\cT_1$ and $\cT_2$ contain $t_1$ and $t_2$ equations and $2t_1 + 1$ and $2t_2 + 1$ variables (respectively), and that $\cT_1 \cap \cT_2$ contains $t$ equations and $v$ variables. Then $\cT_1 \cup \cT_2$ contains exactly $t_1 + t_2 - t$ equations (by inclusion-exclusion), and at most $(2t_1 + 1) + (2t_2 + 1) - v$ variables (since the $v$ variables in $\cT_1 \cap \cT_2$ are all shared by $\cT_1$ and $\cT_2$; the reason we have `at most' is that other variables could be shared as well). 

    Then the fact that $\cT_1 \cup \cT_2$ is $2$-good means that $v \geq 2t + 1$ and \[(2t_1 + 1) + (2t_2 + 1) - v \geq 2(t_1 + t_2 - t) + 1,\] which rearranges to $v \leq 2t + 1$. This means we must have $v = 2t + 1$, so $\cT_1 \cap \cT_2$ is $2$-full. 
\end{proof}

\begin{figure}[ht]
    \begin{tikzpicture}
        \node (1) at (0, 0) {$x_1 - x_2 - x_3 + x_4 = 0$};
        \node (2) at (0, -1) {$x_1 + x_2 - x_5 - x_6 = 0$};
        \node (3) at (0, -2) {$x_1 + x_4 - x_5 - x_7 = 0$};
        \node (4) at (0, 1) {$x_1 + x_7 - x_8 - x_9 = 0$};
        \node (5) at (0, -3) {$x_1 + x_7 - x_{10} - x_{11} = 0$};
        \node (6) at (0, -4) {$x_{10} + x_{11} - x_{12} - x_{13} = 0$};
        \filldraw [MediumPurple3, fill opacity = 0.1, draw opacity = 0.3] (-2.25, -2.35) rectangle (2.25, 1.35);
        \filldraw [DarkSlateGray3, fill opacity = 0.1, draw opacity = 0.3] (-2.25, -4.35) rectangle (2.25, 0.35);
    \end{tikzpicture}

    \caption{Two overlapping $2$-full sets of difference equalities as in Lemma \ref{lem:2-full-intersection}, with $\cT_1$ in purple and $\cT_2$ in blue. Here $\cT_1$ has $4$ equations on $9$ variables and $\cT_2$ has $5$ equations on $11$ variables; their intersection has $3$ equations on $7$ variables, and their union has $6$ equations on $13$ variables.}
\end{figure}

\begin{lemma} \label{lem:box-subbox}
    Let $\cT = \{(*_1), \ldots, (*_t)\}$ be a $2$-good $2$-full minimal implication producing a difference equality $(*)$, and let $\xi_1, \ldots, \xi_t \in \{\pm 1\}$ be such that ${*} = \xi_1{*_1} + \cdots + \xi_t{*_t}$. Let $\cT' = \{(*_1), \ldots, (*_{t'})\}$ for some $t' \leq t$, and suppose that $\cT'$ is $2$-full. Then the equation $(*')$ defined by \[{*'} = \xi_1{*_1} + \cdots + \xi_{t'}{*_{t'}}\] is a difference equality.
\end{lemma}

In other words, Lemma \ref{lem:box-subbox} states that if we have a $2$-full $2$-good minimal implication $\cT$, then any $2$-full subset $\cT' \subseteq \cT$ is itself a minimal implication, and the coefficients it uses to imply $(*')$ match the coefficients of $\cT$ used to imply $(*)$. In particular, if we took $\cT$ and replaced $\cT'$ with $(*')$, then the resulting collection $\{(*'), (*_{t' + 1}), \ldots, (*_t)\}$ would still be a minimal implication producing $(*)$. 

\begin{figure}[ht]
    \centering 
    \begin{tikzpicture}
        \node (1) at (0, 0) {$(*_1) : x_1 - x_2 - x_3 + x_4 = 0$};
        \node (2) at (0, -1) {$(*_2) : x_1 + x_2 - x_5 - x_6 = 0$};
        \node (3) at (0, -2) {$(*_3) : x_1 + x_4 - x_5 - x_7 = 0$};
        \node (4) at (0, -3) {$(*_4) : x_1 + x_7 - x_8 - x_9 = 0$};
        \node [DarkSlateGray3] at (0, -4) {$(*) : x_3 + x_6 - x_8 - x_9 = 0$};
        \node [DarkSlateGray3] at (8, -4) {$(*) : x_3 + x_6 - x_8 - x_9 = 0$};
        \filldraw [DarkSlateGray3, fill opacity = 0.1, draw opacity = 0.3] (-2.75, -3.5) rectangle (2.75, 0.5);
        \filldraw [MediumPurple3, fill opacity = 0.1, draw opacity = 0.3] (-2.5, -2.3) rectangle (2.5, 0.3);
        \node (5) at (8, -1) {$(*') : -x_1 + x_3 + x_6 - x_7 = 0$};
        \node (6) at (8, -3) {$(*_4) : x_1 + x_7 - x_8 - x_9 = 0$};
        \filldraw [DarkSlateGray3, fill opacity = 0.1, draw opacity = 0.3] (8 - 2.75, -3.5) rectangle (8 + 2.75, -0.5);
        \filldraw [MediumPurple3, fill opacity = 0.1, draw opacity = 0.3] (8 - 2.5, -1.3) rectangle (8 + 2.5, -0.7);
        \node [MediumPurple3!50, font = {\Huge}] at (4, -1) {$\rightsquigarrow$};
    \end{tikzpicture}
    \caption{An example of the situation described in Lemma \ref{lem:box-subbox}, with $\cT$ shown in blue and $\cT'$ in purple (we have ${*} = -{*_1} - {*_2} + {*_3} + {*_4} = {*'} + {*_4}$). }
\end{figure}

\begin{proof}
    First, by Lemma \ref{lem:2-good-min-impl}\ref{item:variable-counts}, one variable $x_i$ appears four times among $(*_1)$, \ldots, $(*_t)$, $(*)$, and all other variables in $\cT$ appear exactly twice; we call the variables appearing twice \emph{ordinary}. Note that if an ordinary variable appears twice among $(*_1)$, \ldots, $(*_{t'})$, then it must cancel out of $\xi_1{*_1} + \cdots + \xi_{t'}{*_{t'}}$, which means it cannot appear in $(*')$. Meanwhile, any variable appearing once among $(*_1)$, \ldots, $(*_{t'})$ must appear in $(*')$. 

    In total, $(*_1)$, \ldots, $(*_{t'})$ have $2t' + 1$ variables (because $\cT'$ is $2$-full) and $4t'$ slots for variables to appear in. Since $x_i$ can appear at most four times and every ordinary variable can appear at most twice, there are three ways in which this could occur. 

    \case{1}{$2t' - 1$ variables appear twice in $\cT'$, and two appear once} Then $(*')$ contains at most three variables --- the two variables that appear once in $\cT'$, and possibly $x_i$. This contradicts either the fact that $\cT$ is valid (if $(*')$ contains two variables) or that it is collinearity-free (if $(*')$ contains three variables), so this case is impossible. 
    
    \case{2}{$x_i$ appears three times in $\cT'$, $2t' - 3$ variables appear twice, and three appear once} Then $(*')$ contains exactly four variables --- $x_i$ and the three variables that appear once in $\cT'$. Furthermore, each of these four variables appears exactly once among $(*_{t' + 1})$, \ldots, $(*_t)$, $(*)$ (since $\cT'$ accounts for all but one of their appearances), and we can write \[{*'} = {*} - \xi_{t' + 1}{*_{t' + 1}} - \cdots - \xi_t{*_t}.\] So each of these variables has coefficient $\pm 1$ in $(*')$ (since it appears exactly once on the right-hand side, where it has coefficient $\pm 1$); and since the coefficients of $(*')$ sum to $0$, this means $(*')$ is a difference equality. 

    \case{3}{$x_i$ appears four times in $\cT'$, $2t' - 4$ variables appear twice, and four appear once} Then by again writing ${*'} = {*} - \xi_{t' + 1}{*_{t' + 1}} - \cdots - \xi_t{*_t}$, we can see that $(*')$ must contain exactly four variables (the four which appear once in $\cT'$ --- it cannot contain $x_i$ because none of $(*_{t' + 1})$, \ldots, $(*_t)$, $(*)$ contain $x_i$), each with coefficient $\pm 1$; so $(*')$ is again a difference equality. 
\end{proof}

Finally, we need the observation that we cannot have a $2$-full collection of size $2$ (this is essentially the situation from Example \ref{ex:c-good}\ref{ex-item:collinearity}).

\begin{claim} \label{claim:2-eqns-5-vars}
    Suppose that $(*_1)$ and $(*_2)$ are distinct difference equalities such that $\{(*_1), (*_2)\}$ is valid and collinearity-free. Then $(*_1)$ and $(*_2)$ together contain at least six variables. 
\end{claim}

\begin{proof}
    Assume not, so that $(*_1)$ and $(*_2)$ share at least three variables. We can assume without loss of generality that at least two of these variables have the same signs in $*_1$ and $*_2$ (otherwise two have opposite signs, and we can negate $*_2$ to make them have the same signs). Then if we let $(*)$ be the equation given by \[{*} = {*_1} - {*_2},\] these two variables must cancel out; this means $(*)$ contains at most three variables. 

    But the coefficients of $(*)$ sum to $0$, so if $(*)$ contains two variables then $\{(*_1), (*_2)\}$ is invalid, while if it contains three variables then $\{(*_1), (*_2)\}$ is collinearity-inducing. 
\end{proof}

\section{Certifications by equations containing \texorpdfstring{$x_i$}{xi}} \label{sec:impls-at-i}

In this section, we prove Lemmas \ref{lem:baseline-technical} and \ref{lem:contained-in-s}. Here, we have a collection $\cS$ of at most $s$ independent difference equalities all involving $x_i$, and we want to understand how many pairs $(i, \bullet)$ or $(i', j') \in [i - 1]^2$ it certifies. The crucial input to both proofs is the following lemma, which gives very good control on what minimal implications in $\cS$ look like. 

\begin{lemma}\label{lem:impls-are-small}
    Any subset of $\cS$ which forms a minimal implication has size at most $4$. 
\end{lemma}

\begin{proof}
    Let $\cT = \{(*_1), \ldots, (*_t)\} \subseteq \cS$ be a minimal implication producing a difference equality $(*)$. Then we can write \[{*} = c_1{*_1} + \cdots + c_t{*_t}\] for nonzero $c_1, \ldots, c_t \in \QQ$. Every variable in $\cT$ must appear in at least two of $(*_1)$, \ldots, $(*_t)$, $(*)$, and $x_i$ appears in at least $t$ (it appears in all the equations in $\cT$, and may or may not appear in $(*)$). Meanwhile, each equation only has four slots for variables to appear. So if $\cT$ contains $v$ variables, then we have \[2(v - 1) + t \leq 4(t + 1),\] which rearranges to $v \leq (3t + 6)/2$. Meanwhile, the fact that $\cT$ is $c$-good means that $v \geq ct + 1$. If $c$ is sufficiently close to $2$, this is a contradiction if $t \geq 5$; so we get that $t \leq 4$. 
\end{proof}

For each $t \leq 4$, we define a \emph{$t$-implication} to be a size-$t$ subset of $\cS$ which forms a minimal implication (so Lemma \ref{lem:impls-are-small} means that every difference equality implied by $\cS$ is produced by some $t$-implication). By Claim \ref{claim:c-to-2}, every $t$-implication (and every union of a small number of $t$-implications) is $2$-good; this means we can use the results of Section \ref{sec:observations} to analyze the structure of our $t$-implications.

In particular, Lemma \ref{lem:2-good-min-impl}\ref{item:unique} means that every $t$-implication produces exactly one difference equality, and Lemma \ref{lem:2-good-min-impl}\ref{item:variable-counts} means that this difference equality involves $x_i$ if $t \in \{1, 3\}$ and does not involve $x_i$ if $t \in \{2, 4\}$. 

Of course $1$-implications are trivial (the difference equality produced by $\{(*_1)\}$ is just $(*_1)$ itself). In Subsection \ref{subsec:t-impls}, we will prove a few facts about the structure of $t$-implications for $t \in \{2, 3, 4\}$; in Subsection \ref{subsec:proof-impls-xi}, we will use these facts to deduce Lemmas \ref{lem:baseline-technical} and \ref{lem:contained-in-s}. 

\subsection{Analyzing \texorpdfstring{$t$}{t}-implications} \label{subsec:t-impls} First, a $2$-implication is a collection $\{(*_1), (*_2)\}$ of difference equalities containing $x_i$ which minimally implies a difference equality $(*)$ not containing $x_i$. Such collections are easy to describe --- $(*_1)$ and $(*_2)$ must share exactly one variable other than $x_i$ (they cannot share more than one other variable by Claim \ref{claim:2-eqns-5-vars}), and either this variable must occur with the same sign as $x_i$ in both $(*_1)$ and $(*_2)$, or it must occur with opposite sign as $x_i$ in $(*_1)$ and $(*_2)$. In the first case, we say $(*_1)$ and $(*_2)$ are \emph{difference-aligned}; in the second case, we say $(*_1)$ and $(*_2)$ are \emph{sum-aligned}. 

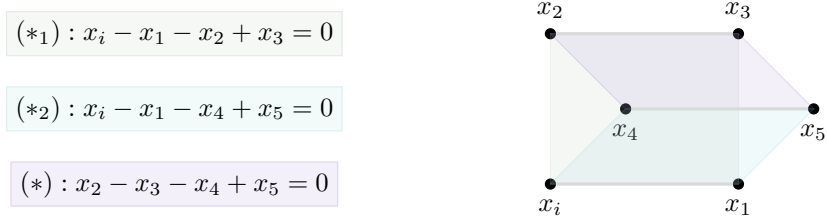
\begin{figure}[ht]
    \begin{tikzpicture}
        \node [fill = Honeydew3!15, draw = Honeydew3!45] at (0, 0) {$(*_1) : x_i - x_1 - x_2 + x_3 = 0$};
        \node [fill = DarkSlateGray3!10, draw = DarkSlateGray3!30] at (0, -1) {$(*_2) : x_i - x_1 - x_4 + x_5 = 0$};
        \node [fill = MediumPurple3!10, draw = MediumPurple3!30] at (0, -2) {$(*) : x_2 - x_3 - x_4 + x_5 = 0$};
        \node [sdot, label = below: {$x_i$}] (i) at (5, -2) {};
        \node [sdot, label = below: {$x_1$}] (1) at (7.5, -2) {};
        \node [sdot, label = above: {$x_2$}] (2) at (5, 0) {};
        \node [sdot, label = above: {$x_3$}] (3) at ($(1) + (2) - (i)$) {};
        \node [sdot, label = below: {$x_4$}] (4) at (6, -1) {};
        \node [sdot, label = below: {$x_5$}] (5) at ($(1) + (4) - (i)$) {};
        \foreach \i\j\k\l\c\p in {i/1/3/2/Honeydew3/0.15, i/1/5/4/DarkSlateGray3/0.1, 2/3/5/4/MediumPurple3/0.1} {
            \filldraw [\c, draw opacity = 0.25, fill opacity = \p] (\i.center) -- (\j.center) -- (\k.center) -- (\l.center) -- cycle;
        }
        \foreach \i\j in {i/1, 2/3, 4/5} {
            \draw [gray!30, very thick] (\i) -- (\j);
        }
    \end{tikzpicture}
    \caption{A difference-aligned $2$-implication, which looks like three pairs with equal differences.}
\end{figure}

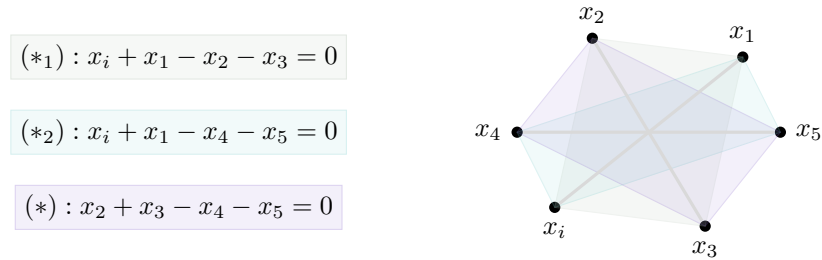
\begin{figure}[ht]
    \begin{tikzpicture}
        \node [fill = Honeydew3!15, draw = Honeydew3!45] at (0, 0) {$(*_1) : x_i + x_1 - x_2 - x_3 = 0$};
        \node [fill = DarkSlateGray3!10, draw = DarkSlateGray3!30] at (0, -1) {$(*_2) : x_i + x_1 - x_4 - x_5 = 0$};
        \node [fill = MediumPurple3!10, draw = MediumPurple3!30] at (0, -2) {$(*) : x_2 + x_3 - x_4 - x_5 = 0$};
        \node [sdot, label = below: {$x_i$}] (i) at (5, -2) {};
        \node [sdot, label = above: {$x_1$}] (1) at (7.5, 0) {};
        \node [sdot, label = above: {$x_2$}] (2) at (5.5, 0.25) {};
        \node [sdot, label = below: {$x_3$}] (3) at ($(1) + (i) - (2)$) {};
        \node [sdot, label = left: {$x_4$}] (4) at (4.5, -1) {};
        \node [sdot, label = right: {$x_5$}] (5) at ($(1) + (i) - (4)$) {};
        \foreach \i\j\k\l\c\p in {i/2/1/3/Honeydew3/0.15, i/4/1/5/DarkSlateGray3/0.1, 2/4/3/5/MediumPurple3/0.1} {
            \filldraw [\c, draw opacity = 0.25, fill opacity = \p] (\i.center) -- (\j.center) -- (\k.center) -- (\l.center) -- cycle;
        }
        \foreach \i\j in {i/1, 2/3, 4/5} {
            \draw [gray!30, very thick] (\i) -- (\j);
        }
    \end{tikzpicture}
    \caption{A sum-aligned $2$-implication, which looks like three pairs with equal sums.}
\end{figure}

We now prove a few facts about $3$-implications. First note that $x_i$ appears in all three equations of a $3$-implication, so by Lemma \ref{lem:2-good-min-impl}\ref{item:variable-counts}, it must be $2$-full (i.e., it contains seven variables). 

\begin{claim} \label{claim:3-impls-cert-5}
    Any $3$-implication certifies at most five pairs $(i, \bullet)$. 
\end{claim}

(When we consider the pairs $(i, \bullet)$ certified by a $3$-implication, we are looking at \emph{all} difference equalities containing $x_i$ that it implies, not necessarily minimally --- this means we consider its three equations as well as the one that it produces.) 

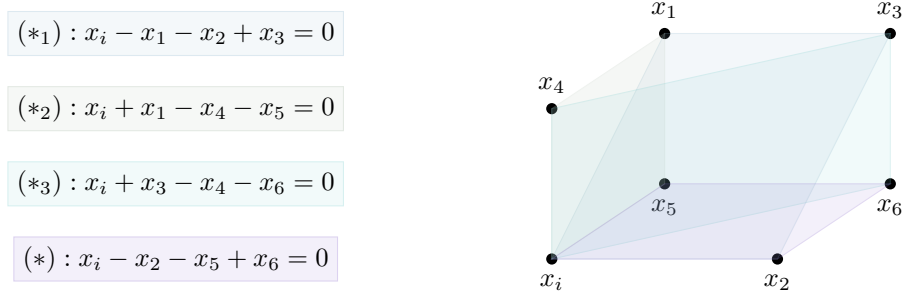
\begin{figure}[ht]
    \begin{tikzpicture}
        \node [fill = LightSkyBlue3!10, draw = LightSkyBlue3!30] at (0, 0) {$(*_1) : x_i - x_1 - x_2 + x_3 = 0$};
        \node [fill = Honeydew3!15, draw = Honeydew3!45] at (0, -1) {$(*_2) : x_i + x_1 - x_4 - x_5 = 0$};
        \node [fill = DarkSlateGray3!10, draw = DarkSlateGray3!30] at (0, -2) {$(*_3) : x_i + x_3 - x_4 - x_6 = 0$};
        \node [fill = MediumPurple3!10, draw = MediumPurple3!30] at (0, -3) {$(*) : x_i - x_2 - x_5 + x_6 = 0$};
        \node [sdot, label = below: {$x_i$}] (1) at (5, -3) {};
        \node [sdot, label = above: {$x_1$}] (2) at (6.5, 0) {};
        \node [sdot, label = below: {$x_2$}] (3) at (8, -3) {};
        \node [sdot, label = above: {$x_3$}] (4) at ($(2) + (3) - (1)$) {};
        \node [sdot, label = above: {$x_4$}] (5) at (5, -1) {};
        \node [sdot, label = below: {$x_5$}] (6) at ($(1) + (2) - (5)$) {};
        \node [sdot, label = below: {$x_6$}] (7) at ($(1) + (4) - (5)$) {};
        \foreach \i\j\k\l\c\p in {1/2/4/3/LightSkyBlue3/0.1, 1/5/2/6/Honeydew3/0.15, 1/5/4/7/DarkSlateGray3/0.1, 1/3/7/6/MediumPurple3/0.1} {
            \filldraw [\c, draw opacity = 0.25, fill opacity = \p] (\i.center) -- (\j.center) -- (\k.center) -- (\l.center) -- cycle;
        }
    \end{tikzpicture}
    \caption{An example of a $3$-implication, where ${*} = {*_1} + {*_2} - {*_3}$. This $3$-implication certifies $(i, 1)$, $(i, 2)$, $(i, 4)$, $(i, 5)$, and $(i, 6)$, but not $(i, 3)$. It looks like a cube with a missing vertex.} \label{fig:3-impl}
\end{figure}

\begin{proof}
    Let the $3$-implication be $\{(*_1), (*_2), (*_3)\}$, and let $(*)$ be the difference equality that it produces. We can assume without loss of generality that $x_i$ has positive coefficient in each of ${*_1}$, \ldots, ${*_3}$, $*$, and that \[{*} = {*_1} + {*_2} - {*_3}.\] Without loss of generality, suppose that \[{*_1} = x_i - x_1 - x_2 + x_3,\] so that $(*_1)$ certifies $(i, 1)$ and $(i, 2)$. If $x_3$ does not appear with opposite sign as $x_i$ in any of the three other equations, then the $3$-implication does not certify $(i, 3)$, and we are done (there are only six variables $x_j$ other than $x_i$ in the $3$-implication, and the $3$-implication can only certify $(i, j)$ for those indices $j$). This scenario is illustrated in Figure \ref{fig:3-impl}. 
    
    Meanwhile, if it does appear with opposite sign as $x_i$ in one of the three equations, that equation has to be $(*_2)$ (it can only appear in one other equation by Lemma \ref{lem:2-good-min-impl}\ref{item:variable-counts}, so if it appeared with opposite sign as $x_i$ in $(*_3)$ or $(*)$, then it would not cancel out of ${*_1} + {*_2} - {*_3} - {*}$). Since $(*_1)$ and $(*_2)$ cannot share more than two variables by Claim \ref{claim:2-eqns-5-vars}, we can assume without loss of generality that \[{*_2} = x_i - x_3 - x_4 + x_5.\] Then the same reasoning shows that the $3$-implication does not certify $(i, 5)$ (as $x_5$ does not appear with opposite sign as $x_i$ in $(*_1)$, and it cannot appear with opposite sign as $x_i$ in $(*_3)$ or $(*)$), so we are done. 
\end{proof}

\begin{claim} \label{claim:3-impls-disjoint}
    Any two distinct $3$-implications are disjoint. 
\end{claim}

\begin{proof}
    Assume for contradiction that we have two $3$-implications which are not disjoint; each $3$-implication is $2$-full, so by Lemma \ref{lem:2-full-intersection} their intersection must be $2$-full as well. But their intersection has size either $1$ or $2$; a set of size $1$ cannot be $2$-full because a single difference equality contains four variables, and a set of size $2$ cannot be $2$-full by Claim \ref{claim:2-eqns-5-vars}. 
\end{proof}

Finally, we need one claim regarding $4$-implications.

\begin{claim} \label{claim:4-impls-overlap}
    If two $4$-implications intersect, then their intersection is a $3$-implication. 
\end{claim}

\begin{figure}[ht]
    \begin{tikzpicture}
        \node [fill = MediumPurple3!10, draw = MediumPurple3!30, align = center] at (-3, -1) {$x_i - x_1 - x_2 + x_3 = 0$ \\ [5pt] $x_i + x_1 - x_4 - x_5 = 0$ \\ [5pt] $x_i + x_3 - x_4 - x_6 = 0$};
        \node [fill = MediumPurple3!10, draw = MediumPurple3!30] at (2, -1) {$x_i - x_2 - x_5 + x_6 = 0$};
        \node at (-3, 0.5) {$x_i - x_2 - x_7 + x_8 = 0$};
        \node at (-3, -2.5) {$x_i + x_6 - x_9 - x_{10} = 0$};
        \node [fill = DarkSlateGray3!10, draw = DarkSlateGray3!30] at (2, 0) {$x_5 - x_6 - x_7 + x_8 = 0$};
        \node [fill = Honeydew3!15, draw = Honeydew3!45] at (2, -2) {$x_2 + x_5 - x_9 - x_{10} = 0$};
        \filldraw [DarkSlateGray3, fill opacity = 0.1, draw opacity = 0.3] (-5, -2) rectangle (-1, 1);
        \filldraw [Honeydew3, fill opacity = 0.15, draw opacity = 0.45] (-5, -3) rectangle (-1, 0);
        \node [font = {\Huge}, MediumPurple3!50] at (-0.5, -1) {$\rightsquigarrow$};
        \node [font = {\Huge}, DarkSlateGray3!50] at (-0.5, 0) {$\rightsquigarrow$};
        \node [font = {\Huge}, Honeydew3!75] at (-0.5, -2) {$\rightsquigarrow$};
    \end{tikzpicture}
    \caption{Two intersecting $4$-implications (shown in blue and green on the left, with the equations they produce on the right) and the $3$-implication formed by their intersection (shown in purple, with the equation it produces on the right).}
\end{figure}
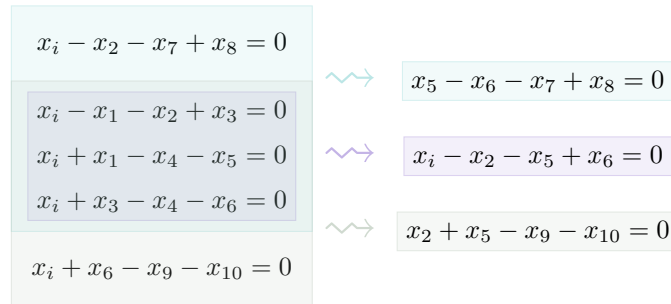

\begin{proof}
    First, $x_i$ appears in all four equations of a $4$-implication, so any $4$-implication is $2$-full. So if two $4$-implications intersect, then by Lemma \ref{lem:2-full-intersection}, their intersection is $2$-full as well. This means their intersection cannot have size $1$ or $2$, so it must have size $3$. Furthermore, Lemma \ref{lem:box-subbox} means that this intersection itself forms a minimal implication, so it must be a $3$-implication. 
\end{proof}

\subsection{Proof of Lemmas \ref{lem:baseline-technical} and \ref{lem:contained-in-s}} \label{subsec:proof-impls-xi}

We are now ready to deduce Lemmas \ref{lem:baseline-technical} and \ref{lem:contained-in-s}. First we prove a stronger version of Lemma \ref{lem:baseline-technical} which gives some information about near-equality cases. 

\begin{claim} \label{claim:baseline-technical-extra}
    The number of pairs $(i, \bullet)$ that $\cS$ certifies is at most $2s$. Furthermore, if this number is at least $2(1 - \eps)s$, then $\cS$ contains at most $2\eps s$ $3$-implications, and every equation in $\cS$ is difference-aligned with at most $8\eps s$ others. 
\end{claim}

\begin{proof}
    Imagine that we write down all equations in $\cS$ (of which there are at most $s$) and draw a \emph{box} around every $3$-implication, as well as around every individual equation not included in any $3$-implication. These boxes partition $\cS$ (since $3$-implications are disjoint by Claim \ref{claim:3-impls-disjoint}), and every pair $(i, \bullet)$ certified by $\cS$ is certified by some box. 

    \begin{figure}[ht]
        \begin{tikzpicture}
            \node at (0, 0) {$x_i - x_1 - x_2 + x_3 = 0$};
            \node at (0, -1) {$x_i + x_1 - x_4 - x_5 = 0$};
            \node at (0, -2) {$x_i + x_3 - x_4 - x_6 = 0$};
            \node at (0, -3) {$x_i - x_7 - x_8 + x_9 = 0$};
            \node at (0, -4) {$x_i + x_7 - x_{10} - x_{11} = 0$};
            \node at (0, -5) {$x_i - x_8 - x_{11} + x_{12} = 0$};
            \node at (0, -6) {$x_i + x_{13} - x_{14} - x_{15} = 0$};
            \node at (0, -7) {$x_i + x_{13} - x_{16} - x_{17} = 0$};
            \node at (0, -8) {$x_i - x_{16} - x_{18} + x_{19} = 0$};
            \filldraw [DarkSlateGray3, draw opacity = 0.3, fill opacity = 0.1] (-2.3, -2.3) rectangle (2.3, 0.3);
            \filldraw [DarkSlateGray3, draw opacity = 0.3, fill opacity = 0.1] (-2.3, -5.3) rectangle (2.3, -2.7);
            \filldraw [DarkSlateGray3, draw opacity = 0.3, fill opacity = 0.1] (-2.3, -6.3) rectangle (2.3, -5.7);
            \filldraw [DarkSlateGray3, draw opacity = 0.3, fill opacity = 0.1] (-2.3, -7.3) rectangle (2.3, -6.7);
            \filldraw [DarkSlateGray3, draw opacity = 0.3, fill opacity = 0.1] (-2.3, -8.3) rectangle (2.3, -7.7);
            \foreach \i\l in {-1/{$1$, $2$, $4$, $5$, $6$}, -4/{$7$, $8$, $10$, $11$, $12$}, -6/{$14$, $15$}, -7/{$16$, $17$}, -8/{$16$, $18$}} {
                \node [MediumPurple3!50, font = {\Huge}] at (3, \i) {$\rightsquigarrow$};
                \node [MediumPurple3, anchor = west] at (3.5, \i) {\l};
            }
        \end{tikzpicture}
        \caption{A possible configuration of boxes in the proof of Claim \ref{claim:baseline-technical-extra}, where the numbers to the right of each box are the indices $j$ for which that box certifies $(i, j)$. For example, for the second box, the equations already present in the box certify $(i, 7)$, $(i, 8)$, $(i, 10)$, and $(i, 11)$; and the box implies $x_i + x_9 - x_{10} - x_{12} = 0$, so it also certifies $(i, 12)$.}
    \end{figure}
    A box of size $1$ (i.e., a lone equation) certifies exactly two pairs $(i, \bullet)$, corresponding to the two variables which appear with opposite sign as $x_i$ in that equation. Meanwhile, a box of size $3$ (i.e., a $3$-implication) certifies at most five pairs $(i, \bullet)$ by Claim \ref{claim:3-impls-cert-5}. 
    
    First, for every box, this means the number of pairs $(i, \bullet)$ it certifies is at most twice its size; so the total number of pairs $(i, \bullet)$ certified by $\cS$ is at most $2s$ (the total size of all boxes). 

    To prove the rest of the claim, suppose that $\cS$ certifies at least $2(1 - \eps)s$ pairs $(i, \bullet)$. First, we can get a loss of one in the above bound for every $3$-implication in $\cS$, since a $3$-implication certifies at most $5$ (rather than $6 = 2 \cdot 3$) pairs. This means $\cS$ must contain at most $2\eps s$ $3$-implications, since we cannot have a loss of more than $2\eps s$. 
    
    Now assume for contradiction that $\cS$ contains an equation $(*)$ which is difference-aligned with more than $8\eps s$ others. First, at most $6\eps s$ equations in $\cS$ are in $3$-implications, so $(*)$ is difference-aligned with more than $2\eps s$ lone equations. But we get a loss of one in the above bound for each such equation (since if $(*)$ is difference-aligned with $(*')$, then one of the pairs $(i, \bullet)$ certified by $(*')$ was already certified by $(*)$). Again this means we get a loss of more than $2\eps s$, which is a contradiction. 
\end{proof}

This completes the proof of Lemma \ref{lem:baseline-technical}. For Lemma \ref{lem:contained-in-s}, we will use Claim \ref{claim:baseline-technical-extra} to show that $\cS$ contains very few $4$-implications or difference-aligned $2$-implications. So the main contribution will come from sum-aligned $2$-implications; and the assumption that $\cS$ does not contain a large star means that the number of sum-aligned $2$-implications is substantially less than $\binom{s}{2}$, as captured by the following claim. 

\begin{claim} \label{claim:sum-aligned}
    If $\cS$ does not imply a star of size at least $2(1 - 16\eps)s$, then every equation in $\cS$ is sum-aligned with at most $(1 - 16\eps)s$ others. 
\end{claim}

\begin{proof}
    Suppose we partition the difference equalities in $\cS$ based on which variable appears with the same sign as $x_i$, so that an equation is sum-aligned with precisely the other equations in its part. Then a part consisting of $p$ equations forms a star of size $2(p + 1)$. (For this, we are using the fact that two equations cannot share more than two variables, which ensures that the other $2p$ variables appearing in equations of the part are all distinct.)

    \begin{figure}[ht]
        \begin{tikzpicture}
            \node at (0, 0) {$x_i + x_1 - x_2 - x_3 = 0$};
            \node at (0, -1) {$x_i + x_1 - x_4 - x_5 = 0$};
            \node at (0, -2) {$x_i + x_1 - x_6 - x_7 = 0$};
            \node at (0, -3) {$x_i - x_4 + x_8 - x_9 = 0$};
            \node at (0, -4) {$x_i + x_8 - x_{10} - x_{11} = 0$};
            \filldraw [DarkSlateGray3, draw opacity = 0.3, fill opacity = 0.1] (-2, -2.25) rectangle (2, 0.25);
            \filldraw [MediumPurple3, draw opacity = 0.3, fill opacity = 0.1] (-2, -4.25) rectangle (2, -2.75);
            \node [sdot, label = left: {$x_i$}] (i) at (4.5, -2) {};
            \node [sdot, label = right: {$x_4$}] (4) at (7, -2) {};
            \node [sdot, label = above: {$x_1$}] (1) at (7.5, -0.25) {};
            \node [sdot, label = left: {$x_2$}] (2) at (4, -1) {};
            \node [sdot, label = right: {$x_3$}] (3) at ($(i) + (1) - (2)$) {};
            \node [sdot, label = left: {$x_5$}] (5) at ($(i) + (1) - (4)$) {};
            \node [sdot, label = above: {$x_6$}] (6) at (6.2, 0) {};
            \node [sdot, label = below: {$x_7$}] (7) at ($(i) + (1) - (6)$) {};
            \node [sdot, label = below: {$x_8$}] (8) at (7, -4) {};
            \node [sdot, label = below: {$x_9$}] (9) at ($(i) + (8) - (4)$) {};
            \node [sdot, label = left: {$x_{10}$}] (10) at (4, -3.25) {};
            \node [sdot, label = right: {$x_{11}$}] (11) at ($(i) + (8) - (10)$) {};
            \foreach \i\j in {i/1, 2/3, 4/5, 6/7} {
                \draw [DarkSlateGray3!50, very thick] (\i) -- (\j);
            }
            \foreach \i\j in {i/8, 4/9, 10/11} {
                \draw [MediumPurple3!50, very thick] (\i) -- (\j);
            }
        \end{tikzpicture}
        \caption{A partition as in the proof of Claim \ref{claim:sum-aligned}, with two stars.}
    \end{figure}
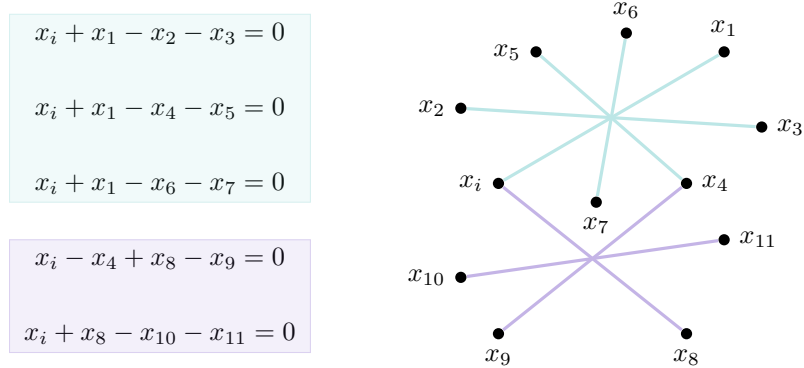

    Since every star has size at most $2(1 - 16\eps)s$, this means every part of the partition contains at most $(1 - 16\eps)s$ equations, as desired. 
\end{proof}

\begin{proof}[Proof of Lemma \ref{lem:contained-in-s}]
    If $\cS$ certifies a pair $(i', j') \in [i - 1]^2$, then this pair must be certified by the product of some $2$-implication or $4$-implication. (This is because by definition, $\cS$ certifies $(i', j')$ if it implies some difference equality where $i'$ is the largest index and $x_{i'}$ and $x_{j'}$ appear with opposite sign; and any such difference equality must be the product of a $2$-implication or $4$-implication.) Furthermore, a single difference equality certifies two pairs $(i', j')$. So in order to prove Lemma \ref{lem:contained-in-s}, it suffices to bound the number of $2$-implications and $4$-implications (since each produces only one difference equality). 

    First, Claim \ref{claim:sum-aligned} means that the number of sum-aligned $2$-implications is at most $(1 - 16\eps)s^2/2$ (since once we have chosen the first equation, there are at most $(1 - 16\eps)s$ choices for the second). Similarly, Claim \ref{claim:baseline-technical-extra} means that every equation in $\cS$ is difference-aligned with at most $8\eps s$ others, so the number of difference-aligned $2$-implications is at most $8\eps s^2/2 = 4\eps s^2$. 

    Next, the number of $4$-implications which are disjoint from all other $4$-implications is at most $s/4$. 

    Finally, to bound the number of $4$-implications which intersect some other $4$-implication, Claim \ref{claim:4-impls-overlap} means that any such $4$-implication consists of a $3$-implication together with one extra equation. By Claim \ref{claim:baseline-technical-extra} there are at most $2\eps s$ $3$-implications, and there are at most $s$ choices for the extra equation, so there are at most $2\eps s^2$ such $4$-implications. 

    Putting these bounds together, the total number of $2$-implications and $4$-implications in $\cS$ is at most \[\frac{(1 - 16\eps)s^2}{2} + 4\eps s^2 + \frac{s}{4} + 2\eps s^2 \leq \frac{(1 - 2\eps)s^2}{2}\] (using the assumption that $\eps$ is small and $s \geq 1/\eps$); each corresponds to two pairs $(i', j') \in [i - 1]^2$ that $\cS$ certifies, so the total number of such pairs is at most $(1 - 2\eps)s^2$. 
\end{proof}

\section{Extra certifications from a small set} \label{sec:intersect-r}

In this section, we prove Lemma \ref{lem:intersect-r}, which states that if we start with a large collection $\cS$ of difference equalities containing $x_i$, then adding a small collection of new difference equalities $\cR$ (satisfying certain conditions) cannot cause too many additional pairs to get certified. 

Imagine that we write down all equations in $\cR$ and $\cS$, and draw a \emph{box} around every minimal implication which intersects $\cR$ (in particular, we draw a box of size $1$ around every equation in $\cR$); then any pair certified by $\cR \cup \cS$ but not $\cS$ alone is certified by some box. We are going to show that the union of all boxes cannot be much bigger than $\cR$, as quantified by the following claim. 

\begin{claim} \label{claim:intersect-r-boxes}
    The union of all boxes has size at most $16\abs{\cR}$. 
\end{claim}

Claim \ref{claim:intersect-r-boxes} directly implies Lemma \ref{lem:intersect-r} --- every difference equality contains only four variables, so Claim \ref{claim:intersect-r-boxes} means the union of all boxes contains at most $64\abs{\cR}$ variables, and therefore certifies at most \[\frac{1}{2} \cdot (64\abs{\cR})^2 = 2048\abs{\cR}^2\] pairs (since if the union of boxes certifies $(i', j')$, then $x_{i'}$ and $x_{j'}$ have to both appear in this union). 

\begin{figure}[ht]
    \begin{tikzpicture}[scale = 1.1]
        \filldraw [fill = DarkSlateGray3!10, draw = DarkSlateGray3!30] (-6, -4.25) rectangle (6, 4.25);
        \filldraw [fill = white, draw = DarkSlateGray3!30] (-1.6, -1.1) rectangle (1.6, 1.1);
        \filldraw [fill = MediumPurple3!10, draw = MediumPurple3!30] (-1.5, -1) rectangle (1.5, 1);
        \foreach \i in {-5, ..., 5} {
            \foreach \j in {1.5, 2.25, 3, 3.75, -1.5, -2.25, -3, -3.75} {
                \draw [DarkSlateGray3, decorate, decoration = {snake, amplitude = 1.5pt}] (\i - 0.28, \j) to (\i + 0.28, \j);
            }
        }
        \foreach \i in {-5, -4, -3, -2, 2, 3, 4, 5} {
            \foreach \j in {0.75, 0, -0.75} {
                \draw [DarkSlateGray3, decorate, decoration = {snake, amplitude = 1.5pt}] (\i - 0.28, \j) to (\i + 0.28, \j);
            }
        }
        \foreach \i in {-1, ..., 1} {
            \foreach \j in {0.75, 0, -0.75} {
                \draw [MediumPurple3, decorate, decoration = {snake, amplitude = 1.5pt}] (\i - 0.28, \j) to (\i + 0.28, \j);
                \filldraw [gray, fill opacity = 0.08, draw opacity = 0.3] (\i - 0.35, \j - 0.15) rectangle (\i + 0.35, \j + 0.15);
            }
        }
        \foreach \i\j\k\l in {-1.45/0.375/0.45/1.75, -0.45/0.5/1.45/2.75, 0.55/-0.25/2.45/1.75, -0.45/0.25/3.45/-1.75, -0.55/0.25/-2.45/-1.75} {
            \filldraw [gray, fill opacity = 0.08, draw opacity = 0.3] (\i, \j) rectangle (\k, \l);
        }
    \end{tikzpicture}
    \caption{A schematic of Claim \ref{claim:intersect-r-boxes}, where $\cS$ is in blue and $\cR$ in purple, squiggly lines represent equations, and boxes are in gray. Intuitively, we want to show that even though $\cS$ may be huge, only a small portion of it --- whose size is comparable to $\cR$ --- `interacts' with $\cR$ (by being in a box).}
\end{figure}
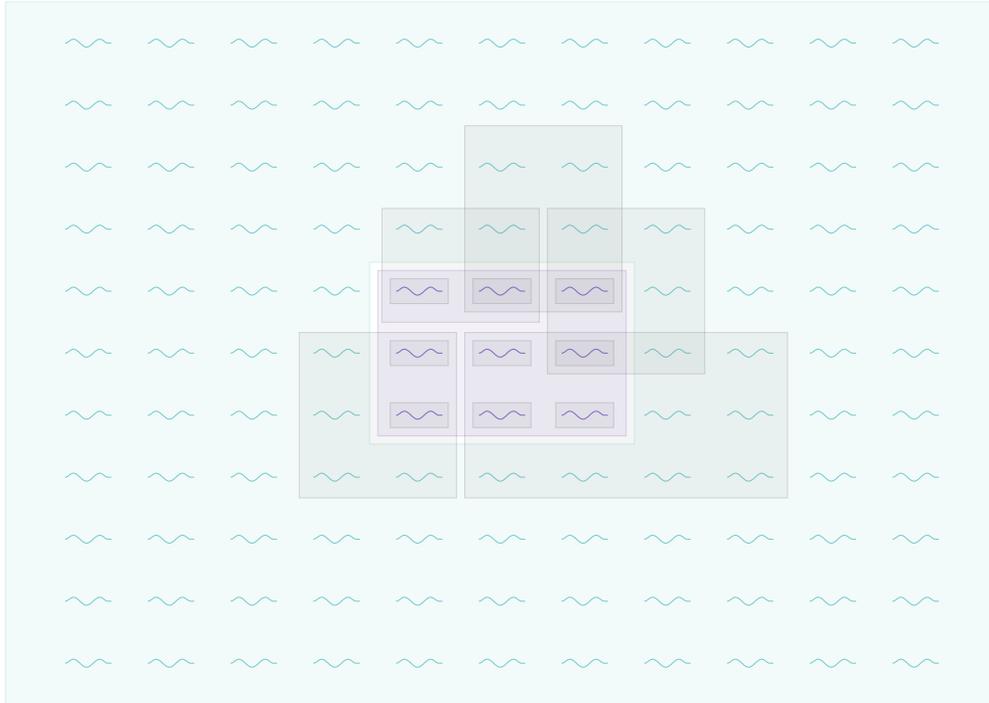

The main idea behind how we prove Claim \ref{claim:intersect-r-boxes} is that we imagine going through the boxes one by one and keeping track of the quantity $v^* - 2t^*$, where $v^*$ is the total number of variables we have seen (other than $x_i$) among the boxes processed so far, and $t^*$ is the total number of equations we have seen. Every time we process a new box, we will gain some number of new variables and equations. We will show that gaining equations in $\cR$ could potentially drive $v^* - 2t^*$ up, but gaining equations in $\cS$ has to drive it \emph{down}. But in the end, $v^* - 2t^*$ cannot be too negative compared to $t^*$ (the fact that $\cR \cup \cS$ is $c$-good means that $v^* \geq ct^*$, and $c$ is very close to $2$). This will mean that the fraction of our equations coming from $\cS$ cannot be too much bigger than the fraction coming from $\cR$, which will imply Claim \ref{claim:intersect-r-boxes}. 

To make this analysis work, we need two lemmas that describe how $v^* - 2t^*$ changes when we add boxes. 

\begin{lemma} \label{lem:indiv-box}
    Let $\cT$ be a box containing $t$ equations, of which $r$ equations are from $\cR$ and $s$ are from $\cS$, and containing $v$ variables which are not $x_i$. Then we have $v - 2t \leq 2r - s/2$. 
\end{lemma}

\begin{proof}
    Let $\cT = \{(*_1), \ldots, (*_t)\}$ and let $(*)$ be a difference equality that $\cT$ produces, so that \[{*} = c_1{*_1} + \cdots + c_t{*_t}\] for nonzero $c_1, \ldots, c_t \in \QQ$. Then each of our $v$ variables has to appear at least twice among $(*_1)$, \ldots, $(*_t)$, $(*)$, while $x_i$ appears exactly $s$ times. (Note that $(*)$ does not contain $x_i$ because we assumed that every difference equality containing $x_i$ implied by $\cR \cup \cS$ is in fact implied by $\cS$ alone, and since $(*)$ comes from a minimal implication which intersects $\cR$, it is not implied by $\cS$ alone.) Since each equation only has four slots for variables to appear, this means \[2v + s \leq 4(t + 1),\] which rearranges to \[v - 2t \leq 2 - \frac{s}{2} \leq 2r - \frac{s}{2}\] (we have $r \geq 1$ because every box intersects $\cR$). 
\end{proof}

\begin{lemma} \label{lem:overlap-box}
    Let $\cT$ be a box, and let $\cT' \subseteq \cT$ be nonempty. Suppose that $\cT$ contains $t$ equations not in $\cT'$, of which $r$ are from $\cR$ and $s$ are from $\cS$. Also suppose that $\cT$ contains $v$ variables which are not $x_i$ and are not contained in $\cT'$. Then we have $v - 2t \leq r/2 - s/4$. 
\end{lemma}

\begin{proof}
    First, if $\cT' = \cT$ then there is nothing to show (all the relevant quantities are $0$). Now assume that $\cT'$ is a proper subset of $\cT$ (i.e., $t > 0$). Let $\cT' = \{(*_1), \ldots, (*_{t'})\}$ and $\cT = \{(*_1), \ldots, (*_{t' + t})\}$, and let $(*)$ be a difference equality that $\cT$ produces (as in the proof of Lemma \ref{lem:indiv-box}, $(*)$ cannot contain $x_i$); this means \[{*} = c_1{*_1} + \cdots + c_{t' + t}{*_{t' + t}}\] for nonzero $c_1, \ldots, c_{t' + t} \in \QQ$. Let $(*')$ be the portion of this linear combination coming from $\cT'$, i.e., \[{*'} = c_1{*_1} + \cdots + c_{t'}{*_{t'}}.\] The coefficients of $(*')$ must sum to $0$; and $(*')$ cannot be identically zero because $\cT$ is independent, it cannot have two variables because $\cT$ is valid, and it cannot have three variables because $\cT$ is collinearity-free. So $(*')$ must contain four variables, and all variables in $(*')$ must be present in $\cT'$. 

    Now we can write 
    \begin{equation}
        {*'} = {*} - c_{t' + 1}{*_{t' + 1}} - \cdots - c_{t' + t}{*_{t' + t}}.\label{eqn:star-prime}
    \end{equation}
    Each of the $v$ variables which are not present in $\cT'$ must appear in at least two equations on the right-hand side of \eqref{eqn:star-prime} (they cannot appear in $(*')$, so they must cancel out of the right-hand side); each of the variables in $(*')$ must appear in at least one equation; and $x_i$ appears in exactly $s$ equations. Meanwhile, there are $t + 1$ equations on the right-hand side of \eqref{eqn:star-prime}, and each has four slots for variables to appear. So this means 
    \begin{equation}
        2v + 4 + (s - 1) \leq 4(t + 1). \label{eqn:t-tprime}
    \end{equation}
    (The reason the third term is $s - 1$ instead of $s$ is because $x_i$ could potentially be one of the four variables in $(*')$.) This rearranges to \[v - 2t \leq \frac{1}{2} - \frac{s}{2}.\] In particular, if $r \geq 1$, or if $r = 0$ and $s \geq 2$, then we are immediately done (as this is at most the desired bound of $r/2 - s/4$). So the only case it remains to consider is when $r = 0$ and $s = 1$. And for this case, it suffices to show that equality does not hold in \eqref{eqn:t-tprime} --- then we get $v - 2t < 0$, so $v - 2t \leq -1 < -1/4$. 

    Assume for contradiction that equality does hold in \eqref{eqn:t-tprime}. Then we need $(*')$ to contain \emph{exactly} four variables (otherwise we could replace the $4$ on the left-hand side with a $5$), and one of those four variables has to be $x_i$ (otherwise we could replace $s - 1$ with $s$). Also, in this case, \eqref{eqn:star-prime} states that \[{*'} = {*} - c_{t' + 1}{*_{t' + 1}}.\] But $(*)$ and $(*_{t' + 1})$ are both difference equalities, so each has exactly four variables. Then for $(*')$ to also contain exactly four variables, we need $(*)$ and $(*_{t' + 1})$ to share exactly two variables (they cannot share more than two variables because of Claim \ref{claim:2-eqns-5-vars}), and these two variables must cancel out of ${*} - {c_{t' + 1}}{*_{t' + 1}}$. In particular, this means $c_{t' + 1} = \pm 1$. But then the four variables in $(*')$ all must have coefficients of $\pm 1$, and since these coefficients sum to $0$, this means $(*')$ is a difference equality. 

    So now $(*')$ is a difference equality containing $x_i$. Since $(*')$ is minimally implied by $\cT'$, and we assumed that all difference equalities containing $x_i$ implied by $\cR \cup \cS$ are implied by $\cS$ alone, we must have $\cT' \subseteq \cS$. Finally, since $r = 0$, this means we have $\cT \subseteq \cS$ as well. But this is a contradiction, because by definition all boxes intersect $\cR$.
    
    We have shown that equality cannot hold in \eqref{eqn:t-tprime} when $r = 0$ and $s = 1$, so we are done. 
\end{proof}

\begin{proof}[Proof of Claim \ref{claim:intersect-r-boxes}]
    Imagine that we process boxes one at a time, and throughout the process, we let $v^*$, $t^*$, $r^*$, and $s^*$ be the total number of non-$x_i$ variables, total equations, equations in $\cR$, and equations in $\cS$ that we have seen so far (among the processed boxes). We claim that at all times, we have 
    \begin{equation}
        v^* - 2t^* \leq 2r^* - \frac{s^*}{4}. \label{eqn:box-processing}
    \end{equation}
    Both sides start out at $0$. Now suppose that \eqref{eqn:box-processing} was true before we added a new box $\cT$; we will show that it remains true after we add $\cT$ as well. 

    \case{1}{$\cT$ does not intersect any previously added boxes} Then we can define $t$, $r$, $s$, and $v$ as in Lemma \ref{lem:indiv-box}. Adding $\cT$ to our picture contributes exactly $t$ new equations, exactly $r$ new equations in $\cR$, and exactly $s$ new equations in $\cS$; and it contributes \emph{at most} $v$ new variables (in other words, $t^*$, $r^*$, and $s^*$ increase by exactly $t$, $r$, and $s$, while $v^*$ increases by at most $v$). And Lemma \ref{lem:indiv-box} gives that \[v - 2t \leq 2r - \frac{s}{2} \leq 2r - \frac{s}{4},\] so the left-hand side of \eqref{eqn:box-processing} increases by at most as much as the right-hand side. 

    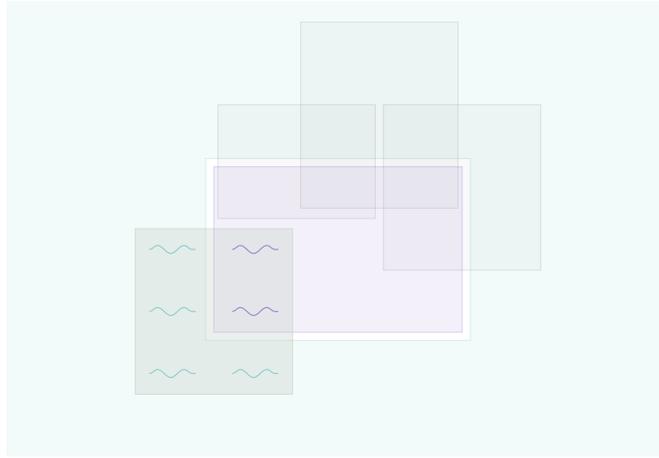
\begin{figure}[ht]
        \begin{tikzpicture}[scale = 1.1]
            \fill [DarkSlateGray3!10] (-4, -2.5) rectangle (4, 3);
            \filldraw [fill = white, draw = DarkSlateGray3!30] (-1.6, -1.1) rectangle (1.6, 1.1);
            \filldraw [fill = MediumPurple3!10, draw = MediumPurple3!30] (-1.5, -1) rectangle (1.5, 1);
            \foreach \i\j in {-2/0, -2/-0.75, -2/-1.5, -1/-1.5} {
                \draw [DarkSlateGray3, decorate, decoration = {snake, amplitude = 1.5pt}] (\i - 0.28, \j) to (\i + 0.28, \j);
            }

            \foreach \i\j in {-1/0, -1/-0.75} {
                \draw [MediumPurple3, decorate, decoration = {snake, amplitude = 1.5pt}] (\i - 0.28, \j) to (\i + 0.28, \j);
            }
            \foreach \i\j\k\l in {-1.45/0.375/0.45/1.75, -0.45/0.5/1.45/2.75, 0.55/-0.25/2.45/1.75} {
                \filldraw [gray, fill opacity = 0.05, draw opacity = 0.2] (\i, \j) rectangle (\k, \l);
            }
            \foreach \i\j\k\l in {-0.55/0.25/-2.45/-1.75} {
                \filldraw [Honeydew3, fill opacity = 0.3, draw opacity = 0.6] (\i, \j) rectangle (\k, \l);
            }
        \end{tikzpicture}
        \caption{In the proof of Claim \ref{claim:intersect-r-boxes}, if we add the box shown in green (having already added the boxes shown in gray), then we are in \casetext{Case 1} with $r = 2$ and $s = 4$ (and $t = 6$).}
    \end{figure}

    \case{2}{$\cT$ does intersect the previously added boxes} Then let $\cT'$ be the intersection of $\cT$ with all previously added boxes, and define $t$, $r$, $s$, and $v$ as in Lemma \ref{lem:overlap-box}. Again, adding $\cT$ to our picture contributes exactly $t$ new equations, $r$ new equations in $\cR$, and $s$ new equations in $\cS$; and it contributes at most $v$ new variables (the variables in $\cT$ which also appear in $\cT'$ are certainly not new). Lemma \ref{lem:overlap-box} gives that \[v - 2t \leq \frac{r}{2} - \frac{s}{4} \leq 2r - \frac{s}{4}.\] So again the left-hand side of \eqref{eqn:box-processing} increases by at most as much as the right-hand side. 

    \begin{figure}[ht]
        \begin{tikzpicture}[scale = 1.1]
            \fill [DarkSlateGray3!10] (-4, -2.5) rectangle (4, 3);
            \filldraw [fill = white, draw = DarkSlateGray3!30] (-1.6, -1.1) rectangle (1.6, 1.1);
            \filldraw [fill = MediumPurple3!10, draw = MediumPurple3!30] (-1.5, -1) rectangle (1.5, 1);
            \foreach \i\j in {0/2.25, 1/2.25} {
                \draw [DarkSlateGray3, decorate, decoration = {snake, amplitude = 1.5pt}] (\i - 0.28, \j) to (\i + 0.28, \j);
            }
            \foreach \i\j in {0/1.5, 1/1.5} {
                \draw [DarkSlateGray3!50, decorate, decoration = {snake, amplitude = 1.5pt}] (\i - 0.28, \j) to (\i + 0.28, \j);
            }

            \foreach \i\j in {0/0} {
                \draw [MediumPurple3, decorate, decoration = {snake, amplitude = 1.5pt}] (\i - 0.28, \j) to (\i + 0.28, \j);
            }
            \foreach \i\j in {1/0, 0/0.75, 1/0.75} {
                \draw [MediumPurple3!50, decorate, decoration = {snake, amplitude = 1.5pt}] (\i - 0.28, \j) to (\i + 0.28, \j);
            }
            \foreach \i\j\k\l in {-1.45/0.5/0.45/1.75, 0.55/-0.25/2.45/1.75} {
                \filldraw [gray, fill opacity = 0.05, draw opacity = 0.2] (\i, \j) rectangle (\k, \l);
            }
            \foreach \i\j\k\l in {-0.45/-0.375/1.45/2.75} {
                \filldraw [Honeydew3, fill opacity = 0.3, draw opacity = 0.6] (\i, \j) rectangle (\k, \l);
            }
        \end{tikzpicture}
        \caption{In the proof of Claim \ref{claim:intersect-r-boxes}, if we add the box shown in green (having already added the boxes shown in gray), then we are in \casetext{Case 2} with $r = 1$ and $s = 2$ (and $t = 3$).}
    \end{figure}
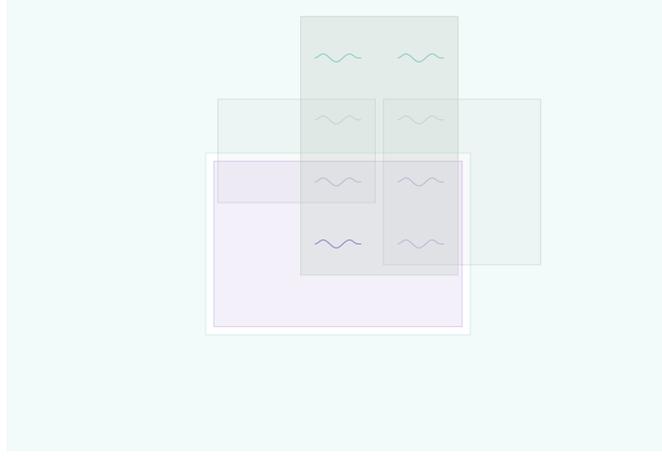
    
    This means \eqref{eqn:box-processing} remains true throughout the process; in particular, it is true at the end, when we have processed all boxes (so $t^*$ is the size of the union of all boxes, and $r^*$ and $s^*$ are the sizes of its intersections with $\cR$ and $\cS$). 

    Now, the fact that $\cR \cup \cS$ is $c$-good means that $v^* \geq ct^*$ (the reason we do not have a $+1$ is because $v^*$ does not count $x_i$). Combining this with \eqref{eqn:box-processing} gives \[(c - 2)t^* \leq v^* - 2t^* \leq 2r^* - \frac{s^*}{4},\] and plugging in $t^* = r^* + s^*$ and rearranging gives \[\left(\frac{1}{4} - (2 - c)\right)s^* \leq (2 + (2 - c))r^*.\] For $c$ sufficiently close to $2$, this means $s^* \leq 15r^*$, so $t^* \leq 16r^* \leq 16\abs{\cR}$. 
\end{proof}

\section{Adding equations to a huge star} \label{sec:huge-star}

In this section, we prove Lemma \ref{lem:outside-star}, which states that if we start with a huge star $\cP$ and add a small collection $\cS$ of additional difference equalities involving $x_i$ (without adding $x_i$ to the star), then the number of pairs $(i, \bullet)$ we certify cannot be much more than the size of $\cS$. This proof has two components. The first says that by enlarging $\cS$ a bit, we can find a (not necessarily independent) set of `representatives' $\cS'$ such that if we want to understand what difference equalities containing $x_i$ are implied by $\cS \cup \cP$, it suffices to consider implications involving only \emph{one} equation from $\cS'$ (\emph{a priori} we would need to consider implications with arbitrarily many equations from $\cS$). The second component handles implications of this simple form. 

\begin{lemma} \label{lem:representatives}
    There exists a set $\cS'$ of difference equalities containing $x_i$ which are implied by $\cS \cup \cP$ such that $\sabs{\cS'} \leq 2\abs{\cS}$, and every difference equality containing $x_i$ which is implied by $\cS \cup \cP$ is in fact implied by $\{(*)\} \cup \cP$ for some $(*) \in \cS'$. 
\end{lemma}

\begin{lemma} \label{lem:indiv-eqn}
    For any difference equality $(*)$ containing $x_i$, the set $\{(*)\} \cup \cP$ certifies at most $3$ pairs $(i, \bullet)$. 
\end{lemma}

We prove Lemma \ref{lem:representatives} in Subsection \ref{subsec:representatives} and Lemma \ref{lem:indiv-eqn} in Subsection \ref{subsec:indiv-eqn}. Together, they immediately imply Lemma \ref{lem:outside-star} --- Lemma \ref{lem:representatives} means that every pair $(i, \bullet)$ certified by $\cS \cup \cP$ is certified by $\{(*)\} \cup \cP$ for some $\{(*)\} \in \cS'$, and Lemma \ref{lem:indiv-eqn} means that the total number of pairs $(i, \bullet)$ certified by sets of this form is at most $3\abs{\cS'} \leq 6\abs{\cS}$. 

\subsection{Finding a set of representatives} \label{subsec:representatives}

In this subsection, we prove Lemma \ref{lem:representatives}. Imagine that we draw a \emph{box} around every subset of $\cS \cup \cP$ which contains more than one equation from $\cS$ and minimally implies a difference equality containing $x_i$. We define the \emph{head} and \emph{tail} of a box as its intersections with $\cS$ and $\cP$.

\begin{figure}[ht]
    \begin{tikzpicture}
        \fill [DarkSlateGray3!5] (-3, 0) rectangle (3, 2.25);
        \node [DarkSlateGray3, anchor = north west] at (-3, 2.25) {$\cS$};
        \fill [MediumPurple3!5] (-3, 0) rectangle (3, -2);
        \node [MediumPurple3, anchor = south west] at (-3, -2) {$\cP$};
        \foreach \i\j in {0/0.5, 0/1, 0/1.5} {
            \draw [DarkSlateGray3, decorate, decoration = {snake, amplitude = 1.5pt}] (\i - 0.75, \j) to (\i + 0.75, \j);
        }
        \foreach \i\j in {0/-0.5, 0/-1} {
            \draw [MediumPurple3, decorate, decoration = {snake, amplitude = 1.5pt}] (\i - 0.75, \j) to (\i + 0.75, \j);
        }
        \draw [gray!50, dashed] (-3, 0) -- (3, 0);
        \filldraw [gray, fill opacity = 0.07, draw opacity = 0.3] (-1, -1.35) rectangle (1, 1.85);
        \filldraw [DarkSlateGray3, fill opacity = 0.07, draw opacity = 0.3] (-0.9, 0.1) rectangle (0.9, 1.75);
        \filldraw [MediumPurple3, fill opacity = 0.07, draw opacity = 0.3] (-0.9, -0.1) rectangle (0.9, -1.25);
        \node [DarkSlateGray3, font = {\Huge}] at (1.3, 1) {$\rightsquigarrow$};
        \node [DarkSlateGray3, anchor = west] at (1.6, 1.05) {head};
        \node [MediumPurple3, font = {\Huge}] at (1.3, -0.75) {$\rightsquigarrow$};
        \node [MediumPurple3, anchor = west] at (1.6, -0.7) {tail};
        \begin{scope}[xshift = 8cm, yscale = 0.75]
            \node at (0, 2.5) {$(*_1) : x_i - x_3 - x_5 + x_7 = 0$};
            \node at (0, 1.5) {$(*_2) : x_i - x_4 + x_{10} - x_{11} = 0$};
            \node at (0, 0.5) {$(*_3) : x_i - x_5 - x_{11} + x_{13} = 0$};
            \node at (0, -0.5) {$(*_4) : x_1 + x_2 - x_3 - x_4 = 0$};
            \node at (0, -1.5) {$(*_5) : x_1 + x_2 - x_9 - x_{10} = 0$};
            \filldraw [gray, fill opacity = 0.07, draw opacity = 0.3] (-2.5, -2) rectangle (2.5, 3);
            \filldraw [DarkSlateGray3, fill opacity = 0.07, draw opacity = 0.3] (-2.4, 0.1) rectangle (2.4, 2.9);
            \filldraw [MediumPurple3, fill opacity = 0.07, draw opacity = 0.3] (-2.4, -0.1) rectangle (2.4, -1.9);
        \end{scope}
        
    \end{tikzpicture}
    \caption{A schematic of a box on the left, and an actual example on the right. (This box minimally implies $(*) : x_i + x_7 - x_9 - x_{13} = 0$, with ${*} = {*_1} + {*_2} - {*_3} - {*_4} + {*_5}$.)}
\end{figure}
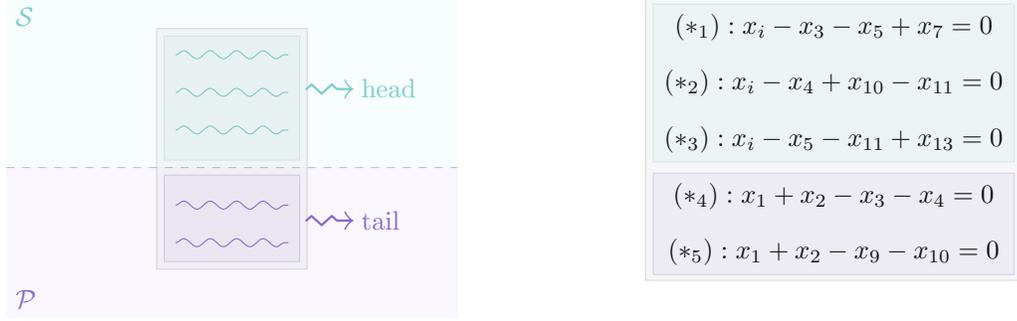

\begin{claim}
    Every box has a head of size exactly $3$ and tail of size at most $2$. 
\end{claim}

\begin{proof}
    Let $\cT = \{(*_1), \ldots, (*_t)\}$ be a box, and let $(*)$ be a difference equality involving $x_i$ that it minimally implies. Let its head and tail have sizes $s$ and $p$, respectively, and suppose that it contains $v$ variables. Then we can write \[{*} = c_1{*_1} + \cdots + c_t{*_t}\] for nonzero $c_1, \ldots, c_t \in \QQ$, so every variable appears at least twice among $(*_1)$, \ldots, $(*_t)$, $(*)$. Furthermore, $x_i$ appears exactly $s + 1$ times, and $x_1$ and $x_2$ each appear at least $p$ times (every equation in $\cP$ contains $x_1 + x_2$). There are $t + 1 = s + p + 1$ equations, and each has four slots for variables to appear, so we have \[2(v - 3) + (s + 1) + 2p \leq 4(s + p + 1),\] which rearranges to \[v \leq \frac{3s + 2p + 9}{2} \leq \frac{3t + 9}{2}.\] The fact that $\cT$ is $c$-good means that $v \geq ct + 1$, so for $c$ sufficiently close to $2$, we get $t \leq 7$. 
    
    Now Claim \ref{claim:c-to-2} means that $\cT$ is actually $2$-good, so we can apply Lemma \ref{lem:2-good-min-impl}\ref{item:variable-counts}. We have $s > 1$ (since by definition, a box contains more than one equation from $\cS$) and $x_i$ appears in exactly $s + 1$ equations among $(*_1)$, \ldots, $(*_t)$, $(*)$. So we must have $s = 3$, and $x_i$ is the variable that appears four times (as given by Lemma \ref{lem:2-good-min-impl}\ref{item:variable-counts}). This means every other variable appears twice; in particular, the fact that $x_1$ and $x_2$ appear twice means that $p \leq 2$. 
\end{proof}

Now Claim \ref{claim:c-to-2} means that every box --- and every union of a small number of boxes --- is $2$-good, so we can use the results of Section \ref{sec:observations}. In particular, Lemma \ref{lem:2-good-min-impl}\ref{item:variable-counts} means that every box is $2$-full (since $x_i$ appears three times in the box), and Lemma \ref{lem:2-good-min-impl}\ref{item:unique} means that every box produces a \emph{unique} difference equality. 

We say a box $\cT$ is \emph{fluffy} if there exists another box $\cT'$ with the same head and strictly smaller tail (meaning that the tail of $\cT'$ is a strict subset of the head of $\cT$). 

\begin{figure}[ht]
    \begin{tikzpicture}
        \fill [DarkSlateGray3!5] (-3, 0) rectangle (3, 2.1);
        \fill [MediumPurple3!5] (-3, 0) rectangle (3, -2);
        \foreach \i\j in {0/0.5, 0/1, 0/1.5} {
            \draw [DarkSlateGray3, decorate, decoration = {snake, amplitude = 1.5pt}] (\i - 0.75, \j) to (\i + 0.75, \j);
        }
        \foreach \i\j in {0/-0.5, 0/-1} {
            \draw [MediumPurple3, decorate, decoration = {snake, amplitude = 1.5pt}] (\i - 0.75, \j) to (\i + 0.75, \j);
        }
        \draw [gray!50, dashed] (-3, 0) -- (3, 0);
        \filldraw [gray, fill opacity = 0.07, draw opacity = 0.3] (-1, -1.35) rectangle (1, 1.85);
        \filldraw [Honeydew3, fill opacity = 0.3, draw opacity = 0.6] (-0.9, -0.75) rectangle (0.9, 1.75);
    \end{tikzpicture}
    \caption{Here the larger box (shown in gray) is fluffy, because the smaller box (shown in green) has the same head and strictly smaller tail.}
\end{figure}
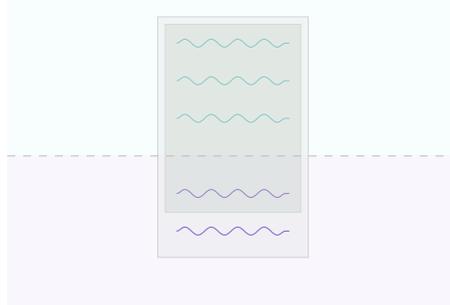

\begin{claim} \label{claim:non-fluffy-disjoint}
    Any two non-fluffy boxes are disjoint. 
\end{claim}

\begin{proof}
    Suppose that $\cT_1$ and $\cT_2$ are distinct boxes which are not disjoint, and without loss of generality assume that $\cT_1$ is not a subset of $\cT_2$; then we will show that $\cT_1$ is fluffy. 

    Let $\cT' = \cT_1 \cap \cT_2$. First, since any two boxes are $2$-full, $\cT'$ is also $2$-full by Lemma \ref{lem:2-full-intersection}. Then Lemma \ref{lem:box-subbox} (applied to $\cT' \subseteq \cT_1$) says that $\cT'$ is itself a minimal implication, implying some difference equality $(*')$. 
    
    Since $\cT'$ is $2$-full, Lemma \ref{lem:2-good-min-impl}\ref{item:variable-counts} means that some variable must appear four times in $\cT' \cup \{(*')\}$, and therefore at least three times in $\cT'$. But the only variable which could possibly appear at least three times in $\cT'$ is $x_i$ (since $x_i$ appears three times in $\cT_1$, and every other variable in $\cT_1$ appears at most twice), and this requires $\cT'$ to contain the full head of $\cT_1$. Furthermore, since this variable has to appear four times in $\cT' \cup \{(*')\}$, it must also appear in $(*')$. 

    So we have shown that $\cT'$ contains the full head of $\cT_1$ and that it forms a minimal implication producing a difference equality $(*')$ containing $x_i$, which means $\cT'$ is also a box. This means $\cT'$ is a box with the same head as $\cT_1$ and strictly smaller tail (it cannot also have the same tail as $\cT_1$ because we assumed $\cT_1 \not\subseteq \cT_2$), showing that $\cT_1$ is fluffy. 
\end{proof}

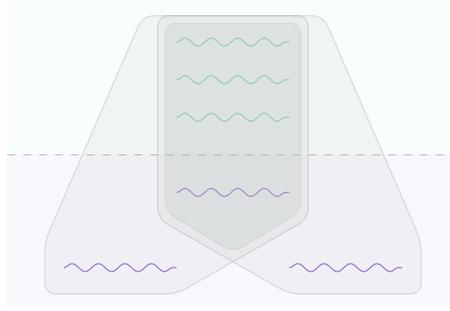
\begin{figure}[ht]
    \begin{tikzpicture}
        \fill [DarkSlateGray3!5] (-3, 0) rectangle (3, 2.1);
        \fill [MediumPurple3!5] (-3, 0) rectangle (3, -2);
        \foreach \i\j in {0/0.5, 0/1, 0/1.5} {
            \draw [DarkSlateGray3, decorate, decoration = {snake, amplitude = 1.5pt}] (\i - 0.75, \j) to (\i + 0.75, \j);
        }
        \foreach \i\j in {0/-0.5, -1.5/-1.5, 1.5/-1.5} {
            \draw [MediumPurple3, decorate, decoration = {snake, amplitude = 1.5pt}] (\i - 0.75, \j) to (\i + 0.75, \j);
        }
        \draw [gray!50, dashed] (-3, 0) -- (3, 0);
        \filldraw [gray, rounded corners, fill opacity = 0.07, draw opacity = 0.3] (1, 1.85) -- (1, -0.85) -- (-0.75, -1.85) -- (-2.5, -1.85) -- (-2.5, -1.15) -- (-1.2, 1.85) -- cycle;
        \begin{scope}[xscale = -1]
            \filldraw [gray, fill opacity = 0.07, rounded corners, draw opacity = 0.3] (1, 1.85) -- (1, -0.85) -- (-0.75, -1.85) -- (-2.5, -1.85) -- (-2.5, -1.15) -- (-1.2, 1.85) -- cycle;
        \end{scope}
        \filldraw [Honeydew3, rounded corners, fill opacity = 0.3, draw opacity = 0.6] (0.9, 1.75) -- (0.9, -0.75) -- (0, -1.3) -- (-0.9, -0.75) -- (-0.9, 1.75) -- cycle;
    \end{tikzpicture}
    \caption{Two intersecting boxes $\cT_1$ and $\cT_2$ shown in gray, and their intersection $\cT'$ shown in green; $\cT'$ shows that $\cT_1$ is fluffy. (In this case $\cT_2$ is also fluffy.)}
\end{figure}

\begin{proof}[Proof of Lemma \ref{lem:representatives}]
    Define $\cS'$ to consist of $\cS$ as well as every difference equality produced by a non-fluffy box. Claim \ref{claim:non-fluffy-disjoint} means that all non-fluffy boxes are disjoint (and each produces only one difference equality), and since each has a head of size $3$, there are at most $\abs{\cS}/3$ of them. So we have $\sabs{\cS'} \leq 4\abs{\cS}/3 \leq 2\abs{\cS}$. 

    To see that $\cS'$ has the desired `representative' property, consider a difference equality $(*)$ containing $x_i$ which is implied by $\cS \cup \cP$, and consider the minimal implication in $\cS \cup \cP$ that produces it. If this minimal implication has exactly one equation from $\cS$, then we are done (because $\cS \subseteq \cS'$). Otherwise this minimal implication is a box $\cT$. 
    
    If $\cT$ is non-fluffy, then $(*)$ is in $\cS'$ and we are done. Otherwise, let $\cT'$ be the smallest box with the same head as $\cT$ and whose tail is contained in that of $\cT$. Then $\cT'$ is non-fluffy, so the difference equality $(*')$ that it produces is in $\cS'$. Furthermore, Lemma \ref{lem:box-subbox} means that $\{(*')\} \cup (\cT \setminus \cT')$ minimally implies $(*)$. And $\cT'$ has the same head as $\cT$, so we have $\cT \setminus \cT' \subseteq \cP$; this means $\{(*')\} \cup \cP$ implies $(*)$, as desired. 
\end{proof}

\begin{figure}[ht]
    \begin{tikzpicture}
        \fill [DarkSlateGray3!5] (-3, 0) rectangle (12, 2.1);
        \fill [MediumPurple3!5] (-3, 0) rectangle (12, -2);
        \foreach \i\j in {0/0.5, 0/1, 0/1.5, 3/1.25, 3/0.75, 7.75/1.25, 7.75/0.75} {
            \draw [DarkSlateGray3, decorate, decoration = {snake, amplitude = 1.5pt}] (\i - 0.75, \j) to (\i + 0.75, \j);
        }
        \foreach \i\j in {0/-0.5, -1.5/-1.5, 1.5/-1.5, 4/-1.5, 7/-1.5, 7.75/-0.75} {
            \draw [MediumPurple3, decorate, decoration = {snake, amplitude = 1.5pt}] (\i - 0.75, \j) to (\i + 0.75, \j);
        }
        \foreach \i\j in {0/2.5, 5.5/2.5, 10/2.5} {
            \draw [Honeydew3, decorate, decoration = {snake, amplitude = 1.5pt}] (\i - 0.75, \j) to (\i + 0.75, \j);
            \draw [Honeydew3, ->] (\i, 1.7) -- (\i, 2.3);
        }
        \draw [gray!50, dashed] (-3, 0) -- (12, 0);
        \filldraw [gray, rounded corners, fill opacity = 0.07, draw opacity = 0.3] (1, 1.85) -- (1, -0.85) -- (-0.75, -1.85) -- (-2.5, -1.85) -- (-2.5, -1.15) -- (-1.2, 1.85) -- cycle;
        \begin{scope}[xscale = -1]
            \filldraw [gray, fill opacity = 0.07, rounded corners, draw opacity = 0.3] (1, 1.85) -- (1, -0.85) -- (-0.75, -1.85) -- (-2.5, -1.85) -- (-2.5, -1.15) -- (-1.2, 1.85) -- cycle;
        \end{scope}
        \filldraw [Honeydew3, rounded corners, fill opacity = 0.3, draw opacity = 0.6] (0.9, 1.75) -- (0.9, -0.75) -- (0, -1.3) -- (-0.9, -0.75) -- (-0.9, 1.75) -- cycle;
        \begin{scope}[xshift = 5.5cm]
            \foreach \i\j in {0/0.5, 0/1, 0/1.5} {
                \draw [DarkSlateGray3, decorate, decoration = {snake, amplitude = 1.5pt}] (\i - 0.75, \j) to (\i + 0.75, \j);
            }
            \foreach \i\j in {0/-0.5} {
                \draw [MediumPurple3, decorate, decoration = {snake, amplitude = 1.5pt}] (\i - 0.75, \j) to (\i + 0.75, \j);
            }
            \filldraw [gray, fill opacity = 0.07, draw opacity = 0.3] (-1, -0.85) rectangle (1, 1.85);
            \filldraw [Honeydew3, fill opacity = 0.3, draw opacity = 0.6] (-0.9, 0.25) rectangle (0.9, 1.75);
        \end{scope}
        \begin{scope}[xshift = 10cm]
            \foreach \i\j in {0/0.5, 0/1, 0/1.5} {
                \draw [DarkSlateGray3, decorate, decoration = {snake, amplitude = 1.5pt}] (\i - 0.75, \j) to (\i + 0.75, \j);
            }
            \foreach \i\j in {0/-0.5, 0/-1} {
                \draw [MediumPurple3, decorate, decoration = {snake, amplitude = 1.5pt}] (\i - 0.75, \j) to (\i + 0.75, \j);
            }
            \filldraw [Honeydew3, fill opacity = 0.3, draw opacity = 0.6] (-1, -1.35) rectangle (1, 1.85);
        \end{scope}
    \end{tikzpicture}
    \caption{In this configuration of boxes (with fluffy boxes shown in gray and non-fluffy boxes in green), $\cS'$ would consist of the $13$ equations in $\cS$ (shown as blue squiggles) and the three equations produced by the fluffy boxes (shown as green squiggles).}
\end{figure}
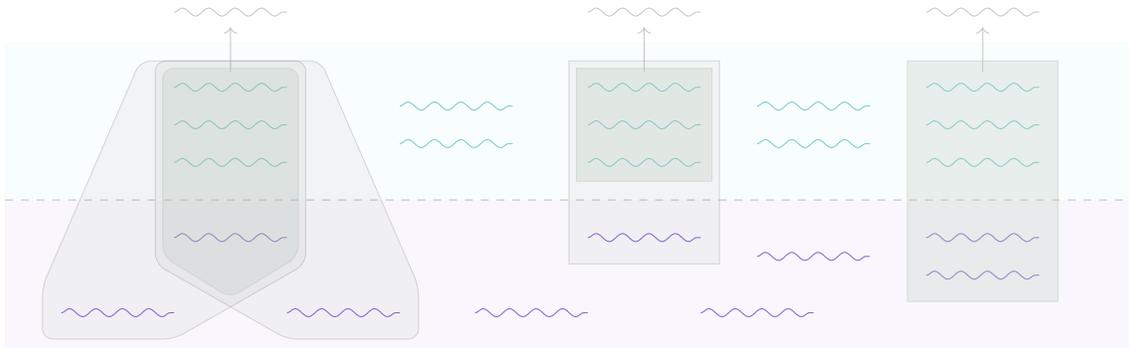

\subsection{Adding a single equation to a star} \label{subsec:indiv-eqn}

In this subsection, we prove Lemma \ref{lem:indiv-eqn}. First note that any equation implied by $\cP$ is of the form 
\begin{equation}
    \alpha_1(x_1 + x_2) + \alpha_2(x_3 + x_4) + \cdots + \alpha_p(x_{2p - 1} + x_{2p}) = 0 \label{eqn:impl-by-p}
\end{equation} 
for $\alpha_1 + \cdots + \alpha_p = 0$. So we want to consider all ways to obtain a difference equality $(*')$ by adding $(*)$ to an equation of this form; and our goal is to show that at most three variables $x_j$ can appear with opposite sign as $x_i$ in such an equation $(*')$. (\emph{A priori} we could be allowed to scale $(*)$ as well when taking our linear combination, but since $x_i$ must have coefficient $\pm 1$ in both $(*)$ and $(*')$, we cannot scale except by $\pm 1$.)

For each $1 \leq j \leq p$, we say $x_{2j - 1}$ and $x_{2j}$ are \emph{opposites}. We now perform casework based on how many variables in $(*)$ are part of the star (i.e., are among $x_1$, \ldots, $x_{2p}$). 

\case{1}{The three variables in $(*)$ are all part of the star} First we claim that none of these three variables can be opposites. Assume for contradiction that two are; then without loss of generality, we can assume $(*)$ contains $x_1$, $x_2$, $x_3$, and $x_i$. But then $\cS \cup \cP$ implies both $(*)$ and the difference equality $x_1 + x_2 = x_3 + x_4$, which share three variables; this contradicts Claim \ref{claim:2-eqns-5-vars}. 

Now we can assume without loss of generality that $(*)$ is the equation \[x_i - x_1 - x_3 + x_5 = 0.\] Then $(*')$ must contain either $-x_1$ or $+x_2$, either $-x_3$ or $+x_4$, and either $+x_5$ or $-x_6$. (Here and in the following cases, when we talk about the signs of coefficients in $(*')$, we assume that $x_i$ has coefficient $+1$.) In particular, this accounts for all three variables other than $x_i$ that appear in $(*')$, so no variables other than these six can appear in $(*')$; and of these six, only $x_1$, $x_3$, and $x_6$ can appear with coefficient $-1$. So $\{(*)\} \cup \cP$ can only certify the three pairs $(i, 1)$, $(i, 3)$, and $(i, 6)$. 

\case{2}{Exactly two variables in $(*)$ are part of the star, and they are opposites} Then we can assume without loss of generality that $(*)$ contains the variables $x_i$, $x_1$, $x_2$, and $x_j$ for some $j > 2p$. It cannot be the equation $x_i - x_1 - x_2 + x_j = 0$ (because we assumed $\cS \cup \cP$ does not imply any equation of this form), so we can assume without loss of generality that it is \[x_i - x_1 - x_j + x_2 = 0.\] Then $x_i$ and $x_j$ both have to appear in $(*')$. If one of $x_1$ and $x_2$ did not appear in $(*')$, then the other would have coefficient $\pm 2$, which is not allowed. So both must appear, and they must have the same coefficients as in $(*)$. This means $(*')$ is the same as $(*)$; so $\{(*)\} \cup \cP$ can only certify $(i, 1)$ and $(i, j)$ (if $i > j$).

\case{3}{Exactly two variables in $(*)$ are part of the star, and they are not opposites} Then we can assume those two variables are $x_1$ and $x_3$, so $(*)$ is either of the form \[x_i - x_1 - x_3 + x_j = 0 \quad \text{or} \quad x_i - x_1 + x_3 - x_j = 0\] for some $j > 2p$. In the first case, $(*')$ must contain $+x_j$, either $-x_1$ or $+x_2$, and either $-x_3$ or $+x_4$. It also contains $+x_i$, and it cannot contain three variables with coefficient $+1$; so $(*')$ has to be the same as $(*)$, and $\{(*)\} \cup \cP$ can only certify $(i, 1)$ and $(i, j)$. 

In the second case, $(*')$ must contain $-x_j$, either $-x_1$ or $+x_2$, and either $+x_3$ or $-x_4$ (in addition to $+x_i$). So the only possibility for $(*')$ other than $(*)$ itself is \[x_i + x_2 - x_4 - x_j = 0.\] This means $\{(*)\} \cup \cP$ can only certify $(i, 1)$, $(i, 4)$, and $(i, j)$. 

\case{4}{At most one variable in $(*)$ is part of the star} Then the three variables in $(*)$ which are not part of the star, including $x_i$, must all appear in $(*')$ with the same coefficients as in $(*)$. This means the fourth must as well --- getting rid of it by adding an equation of the form \eqref{eqn:impl-by-p} would introduce at least three additional variables (which are part of the star but not contained in $(*)$), which would cause $(*')$ to have more than four variables. So $\{(*)\} \cup \cP$ again does not imply any difference equalities containing $x_i$ other than $(*)$ itself, which means it certifies at most two pairs $(i, \bullet)$. 

So in all cases, $\{(*)\} \cup \cP$ certifies at most three pairs $(i, \bullet)$, as desired. 

\section{Adaptation to odd \texorpdfstring{$k$}{k}} \label{sec:odd-k}

We have now completed the proof of Theorem \ref{thm:main}. In this section, we briefly explain how to adapt this proof to get Proposition \ref{prop:odd-k}. For this, we use the same random construction (in other words, we use Lemma \ref{lem:random-constr} directly); meanwhile, we need to replace Lemma \ref{lem:good-certify} with the following statement. 

\begin{lemma} \label{lem:good-certify-odd}
    Suppose that $c$ is sufficiently close to $2$ and that $k$ is odd. Then every $c$-good $k$-configuration certifies at most $(k - 1)(k - 3)/4 + 3$ pairs.
\end{lemma}

This bound is sharp --- the $k$-configuration \[\{x_1 + x_2 = x_3 + x_4 = \cdots = x_{k - 2} + x_{k - 1}, \, x_k - x_1 = x_3 - x_5\},\] consisting of a star of size $k - 1$ and one extra equation involving $x_k$, certifies exactly $(k - 1)(k - 3)/4 + 3$ pairs (the ones certified by the star, as well as $(k, 1)$, $(k, 3)$, and $(k, 6)$). 

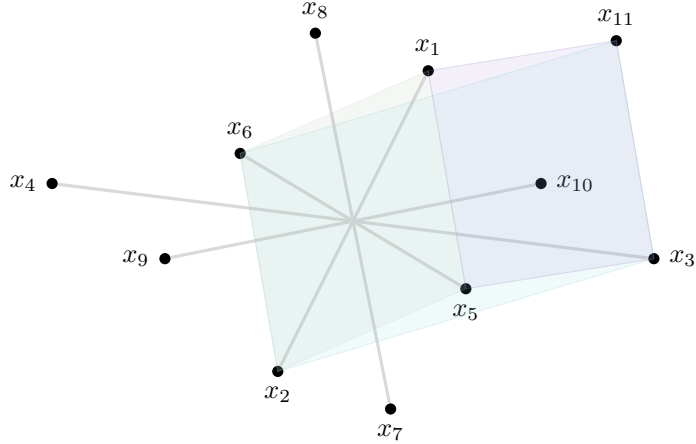
\begin{figure}[ht]
    \begin{tikzpicture}
        \node [sdot, label = above: {$x_1$}] (1) at (1, 2) {};
        \node [sdot, label = right: {$x_3$}] (3) at (4, -0.5) {};
        \node [sdot, label = below: {$x_5$}] (5) at (1.5, -0.9) {};
        \node [sdot, label = below: {$x_7$}] (7) at (0.5, -2.5) {};
        \node [sdot, label = left: {$x_9$}] (9) at (-2.5, -0.5) {};
        \foreach \i\j\k in {1/2/below, 3/4/left, 5/6/above, 7/8/above, 9/10/right} {
            \node [sdot, label = \k: {$x_{\j}$}] (\j) at ($(0, 0) - (\i)$) {};
            \draw [gray!30, very thick] (\i) -- (\j);
        }
        \node [sdot, label = above: {$x_{11}$}] (k) at ($(1) + (3) - (5)$) {};
        \foreach \i\j\k\l\c\p in {k/1/5/3/MediumPurple3/0.1, k/6/2/3/DarkSlateGray3/0.1, 1/5/2/6/Honeydew3/0.15} {
            \filldraw [\c, draw opacity = 0.25, fill opacity = \p] (\i.center) -- (\j.center) -- (\k.center) -- (\l.center) -- cycle;
        }
    \end{tikzpicture}
    \caption{An illustration of the equality case in Lemma \ref{lem:good-certify-odd} for $k = 11$.}
\end{figure}

The proof of Lemma \ref{lem:good-certify-odd} is mostly the same as the proof of Lemma \ref{lem:good-certify} given in Section \ref{sec:backbone}, but it requires one extra ingredient for the huge-star case. 

\begin{lemma} \label{lem:huge-star-only-one}
    Suppose that $\cC$ implies a star $\{x_1 + x_2 = x_3 + x_4 = \cdots = x_{2p - 1} + x_{2p}\}$, and a difference equality $(*)$ on $x_1$, \ldots, $x_{2p + 1}$ which contains $x_{2p + 1}$. Then every other difference equality on $x_1$, \ldots, $x_{2p + 1}$ which contains $x_{2p + 1}$ and is implied by $\cC$ is actually implied by $(*)$ combined with the star. 
\end{lemma}

\begin{proof}
    Assume for contradiction that $\cC$ implies some difference equality $(*')$ on $x_1$, \ldots, $x_{2p + 1}$ which contains $x_{2p + 1}$, such that $(*')$ is not implied by $(*)$ together with the star. 

    Consider all variables which appear in $(*)$ or $(*')$ other than $x_{2p + 1}$, as well as their opposites on the star (where we say $x_{2j - 1}$ and $x_{2j}$ are opposites for $1 \leq j \leq p$). Together with $x_{2p + 1}$, this forms a set of $2t + 1$ variables for some $t \leq 6$ (since $(*)$ and $(*')$ together contain at most six variables other than $x_{2p + 1}$). 

    But the star implies $t - 1$ independent equations on these variables, namely the equations stating that each pair of opposites has equal sum. (For example, if $(*)$ and $(*')$ were $x_{2p + 1} - x_1 - x_3 + x_5 = 0$ and $x_{2p + 1} - x_2 - x_8 + x_{10} = 0$, then we would consider the $11$ variables $x_1$, \ldots, $x_{10}$, $x_{2p + 1}$, and these $t - 1 = 4$ equations would be $x_1 + x_2 = x_3 + x_4$, $x_1 + x_2 = x_5 + x_6$, $x_1 + x_2 = x_7 + x_8$, and $x_1 + x_2 = x_9 + x_{10}$.)
    
    So together with $(*)$ and $(*')$, we get a collection of $(t - 1) + 2 = t + 1$ independent equations on $2t + 1$ variables. But because $\cC$ is $c$-good, any $t + 1$ independent equations that it implies must contain at least $c(t + 1) + 1$ variables. Since $t \leq 6$, this is a contradiction (for $c$ sufficiently close to $2$). 
\end{proof}

\begin{proof}[Proof of Lemma \ref{lem:good-certify-odd}]
    We define parameters $\eps$, $k_0$, and $c$ in the same way as in Section \ref{sec:backbone} --- we choose $\eps$ to be a small absolute constant, $k_0$ to be large with respect to $\eps$, and $c$ such that $2 - c$ is small with respect to $k_0$ and $\eps$. (As in Section \ref{sec:backbone}, $\eps$ is the error parameter for the stability argument, and $k_0$ quantifies what we mean when we say that $k$ is large with respect to $\eps$.)

    First, if $k < k_0$, then every $c$-good $k$-configuration is also $2$-good (this follows from the same argument as in the proof of Claim \ref{claim:c-to-2}, since we chose $2 - c$ to be small with respect to $k_0$). And \cite{Das23} proved that Lemma \ref{lem:good-certify-odd} holds for $2$-good $k$-configurations (this is \cite[Lemma 3.1]{Das23}, and is where the bound \eqref{eqn:das-odd-quadratic-upper} comes from). So from now on, we will assume that $k \geq k_0$, i.e., that $k$ is large with respect to $\eps$. 

    As in Section \ref{sec:backbone}, we split into cases based on whether $\cC$ does or does not imply a huge star, specifically a star of size at least $(1 - 17\eps)k$. 
    
    If $\cC$ does not imply any star of size at least $(1 - 17\eps)k$, then we can use Lemma \ref{lem:stability} directly (Lemma \ref{lem:stability} does not require $k$ to be even); this gives that $\cC$ certifies at most \[\left(1 - \frac{\eps^2}{8}\right)\frac{k^2}{c^2} \leq \frac{(k - 1)(k - 3)}{4} + 3\] pairs (since $k$ is large and $2 - c$ is small with respect to $\eps$). 
    
    Now suppose that $\cC$ does imply a huge star; let $2p$ be the size of the largest star that $\cC$ implies, so that $(1 - 17\eps)k \leq 2p \leq k - 1$. By renaming variables, we can assume that this star is \[\{x_1 + x_2 = x_3 + x_4 = \cdots = x_{2p - 1} + x_{2p}\}.\] Let $\cP = \{x_1 + x_2 = x_3 + x_4, \, x_1 + x_2 = x_5 + x_6, \, \ldots, \, x_1 + x_2 = x_{2p - 1} + x_{2p}\}$, so that $\cP$ is a collection of difference equalities defining this star. 
    
    First, by Lemma \ref{lem:within-star}, the only pairs $(i, j) \in [2p]^2$ that $\cC$ certifies are the ones certified by the star itself, and there are $p^2 - p$ such pairs. 
    
    Next, Lemmas \ref{lem:indiv-eqn} and \ref{lem:huge-star-only-one} together mean that $\cC$ certifies at most three pairs $(2p + 1, \bullet)$. Explicitly, if $\cC$ certifies at least one pair $(2p + 1, \bullet)$, then it implies some difference equality $(*)$ on $x_1$, \ldots, $x_{2p + 1}$ containing $x_{2p + 1}$. Then Lemma \ref{lem:huge-star-only-one} means that \emph{every} difference equality of this form implied by $\cC$ is actually implied by $\{(*)\} \cup \cP$; and Lemma \ref{lem:indiv-eqn} means that $\{(*)\} \cup \cP$ certifies at most three pairs $(2p + 1, \bullet)$. 

    Finally, we can bound the number of pairs $(i, \bullet)$ certified by $\cC$ for each $i \geq 2p + 2$ in the same way as in the proof of Lemma \ref{lem:huge-star} (the huge-star case when $k$ is even): Fix $i$, and let $\cS$ be a maximal collection of difference equalities on $x_1$, \ldots, $x_i$ containing $x_i$ such that $\cS \cup \cP$ is independent. The fact that $\cC$ is $c$-good means that $c(\abs{\cS} + \abs{\cP}) + 1 \leq k$, and therefore \[\abs{\cS} \leq \frac{k - 1}{c} - \frac{(1 - 17\eps)k}{2} + 1 \leq 10\eps k \leq 24\eps p.\] So Lemma \ref{lem:outside-star} gives that $\cS \cup \cP$ (and therefore $\cC$) certifies at most $144\eps p \leq p/2$ pairs $(i, \bullet)$. 

    Finally, putting these bounds together, the total number of pairs that $\cC$ certifies is at most \[(p^2 - p) + 3 + (k - 2p - 1) \cdot \frac{p}{2} = p\left(\frac{k - 3}{2}\right) + 3,\] and plugging in $p \leq (k - 1)/2$ gives the desired bound. 
\end{proof}

\section*{Acknowledgements}

The author thanks Noah Kravitz, Joe Gallian, and Colin Defant for helpful advice and feedback. This work originated from a project the author worked on at the University of Minnesota Duluth REU in 2023; the author is grateful to Jane Street Capital, the National Security Agency, and the CYAN Undergraduate Mathematics Fund at MIT for funding during that time. 

\bibliography{refs}{}
\bibliographystyle{plain}

\appendix 

\section{A modification of Behrend's construction} \label{app:modified-behrend}

In this section, we prove Lemma \ref{lem:modified-behrend}, which states that we can construct subsets of $\{1, 2, \ldots, n\}$ of size $n^{1 - o(1)}$ avoiding all solutions to $\alpha s_1 + \beta s_2 + \gamma s_3 = 0$ for `small' integers $\alpha$, $\beta$, and $\gamma$ with sum $0$. We will use the same construction that Behrend \cite{Beh46} used to produce large $3$-AP-free sets --- taking a high-dimensional sphere and projecting it down to $\ZZ$. The main idea is that solutions to $\alpha s_1 + \beta s_2 + \gamma s_3 = 0$ will correspond to collinear triples on the sphere, which cannot exist. 

Assume that $n$ is sufficiently large, and fix parameters $d = \sfloor{\sqrt{\log n}}$ and $m = \sfloor{e^{\sqrt{\log n}}/16\kappa}$. 

First, for any $\vec{v} = (v_1, \ldots, v_d) \in [m]^d$, we have $v_1^2 + \cdots + v_d^2 \leq dm^2$. So we can choose some $1 \leq r \leq dm^2$ for which the set \[\mathbb{S} = \{\vec{v} \in [m]^d \mid v_1^2 + \cdots + v_d^2 = r\}\] has size at least $m^d/dm^2 = m^{d - 2}/d$. Now let $\varphi \colon \ZZ^d \to \ZZ$ be the map \[\vec{v} \mapsto v_1 + (16\kappa m)v_2 + (16\kappa m)^2v_3 + \cdots + (16\kappa m)^{d - 1}v_d.\] Then $\varphi$ is injective on $[m]^d$ (by the uniqueness of base-$16\kappa m$ expansion), and for all $\vec{v} \in [m]^d$, we have \[1 \leq \varphi(\vec{v}) \leq (16\kappa m)^d \leq e^{\sqrt{\log n} \cdot \sqrt{\log n}} = n.\] 

Now take $S$ to be $\varphi(\mathbb{S})$. First we will show that $S$ does avoid the linear patterns that we wish to avoid. Assume not; this means there exist distinct points $\vec{u}, \vec{v}, \vec{w} \in \mathbb{S}$ such that \[\alpha \varphi(\vec{u}) + \beta \varphi(\vec{v}) + \gamma \varphi(\vec{w}) = 0\] for some nonzero $\alpha, \beta, \gamma \in \ZZ$ of magnitude at most $\kappa$ with sum $0$. Plugging in the definition of $\varphi$, this means \[\sum_{i = 1}^d (16\kappa m)^{i - 1}(\alpha u_i + \beta v_i + \gamma w_i) = 0.\] But we have $-3\kappa m \leq \alpha u_i + \beta v_i + \gamma w_i \leq 3\kappa m$ for each $i$, so by the uniqueness of base-$16\kappa m$ expansion, this means we must have \[\alpha u_i + \beta v_i + \gamma w_i = 0\] for all $i$. But this means $\vec{u}$, $\vec{v}$, and $\vec{w}$ are collinear, which is impossible because $\vec{u}$, $\vec{v}$, and $\vec{w}$ lie on a sphere (and a line can only intersect a given sphere at most twice). 

So we have shown that $S$ indeed avoids the forbidden linear patterns; it remains to show that $\abs{S} = n^{1 - o(1)}$, or equivalently that $\log \abs{S} = (1 - o(1))\log n$. For this, we have $\abs{S} = \abs{\mathbb{S}} \geq m^{d - 2}/d$, which means \[\log \abs{S} \geq (d - 2)\log m - \log d.\] Since $d \geq \sqrt{\log n} - 1$ and $m \geq e^{\sqrt{\log n}}/32\kappa$, we can write $(d + 1)\log(32\kappa m) \geq \sqrt{\log n} \cdot \sqrt{\log n} = \log n$, so 
\begin{align*}
    \log \abs{S} &\geq (d + 1)\log(32\kappa m) - (d + 1)\log 32\kappa - 3\log m - \log d \\
    &\geq \log n - (d + 1)\log(32\kappa) - 3\log m - \log d.
\end{align*} 
Each of the terms being subtracted is $O(\sqrt{\log n})$ (with the implicit constant depending on $\kappa$), so $\log \abs{S} = \log n - O(\sqrt{\log n}) = (1 - o(1))\log n$, as desired. 
\end{document}